\documentclass[11pt,reqno]{amsart}
\usepackage{amsmath,amssymb,latexsym,esint,cite,mathrsfs}
\usepackage{verbatim,wasysym}
\usepackage[left=2.6cm,right=2.6cm,top=2.8cm,bottom=2.8cm]{geometry}
\usepackage{tikz,enumitem,graphicx, subfig, microtype, color}
\usepackage{epic,eepic}
\usepackage{amsmath}
\allowdisplaybreaks
\usepackage[colorlinks=true,urlcolor=blue, citecolor=red,linkcolor=blue,
linktocpage,pdfpagelabels,bookmarksnumbered,bookmarksopen]{hyperref}
\usepackage[hyperpageref]{backref}
\usepackage[english]{babel}

\numberwithin{equation}{section}

\newtheorem{thm}{Theorem}[section]
\newtheorem{lem}[thm]{Lemma}
\newtheorem{cor}[thm]{Corollary}

\newcommand{\N}{\mathbb{N}}

\newcommand{\R}{\mathbb{R}}

\title[1-dimensional symmetric stable operator]{On elliptic equations with N-independent stable operators}

\author[L. Du]{Lele Du}
\author[M. Yang]{Minbo Yang}

\address{Lele Du \newline\indent Dipartimento di Matematica Guido Castelnuovo,\newline\indent Sapienza Universit\`{a} di Roma, Piazzale Aldo Moro 5, 00185 Rome, Italy.}

\address{Riemann International School of Mathematics,\newline\indent Villa Toeplitz via G.B. Vico 46, 21100, Varese, Italy.}

\email{ lele.du@uniroma1.it}	

\address{Minbo Yang \newline\indent School of Mathematical Sciences, Zhejiang Normal University, \newline\indent Jinhua, Zhejiang, 321004, People's Republic of China}
\email{ mbyang@zjnu.edu.cn}

\makeatletter
\@namedef{subjclassname@2020}{\textup{2020} Mathematics Subject Classification}
\makeatother



\thanks{$^\ddag$Minbo Yang is the corresponding author. }
\begin{document}

\begin{abstract}
We investigate the existence and nonexistence of positive solutions of the semilinear elliptic equation
\begin{align*}
\sum^{N}_{i=1}\left(-\partial_{ii}\right)^{s}u=u^{p}
\end{align*}
with N-independent 1-dimensional symmetric $2s$-stable operators. Firstly, in the whole space $\R^{N}$, we establish the nonexistence of positive supersolutions for $1<p\leq\frac{N}{N-2s}$. Furthermore, the symmetry of positive solutions is obtained when $p>\frac{N}{N-2s}$. It is crucial for these solutions to exhibit suitable decay at infinity to compensate for the absence of the Kelvin transform. Notably, while these solutions are symmetric, they are not radially symmetric due to the non-rotational invariance of the operator involved. Next, in the half space $\R_{+}^{N}$, we observe the nonexistence of positive supersolutions in the region $1<p\leq\frac{N+s}{N-s}$. Additionally, we find that positive solutions with appropriate decay for the Dirichlet boundary problem do not exist. Finally, we present the symmetry of positive solutions in unit ball $B_{1}$.
\end{abstract}

\maketitle 

\allowdisplaybreaks

\begin{center}
\begin{minipage}{14cm}
\small
\tableofcontents
\end{minipage}
\end{center}

\newpage

\section{Introduction}
We are interested in the Liouville type results of the positive solutions of the semilinear elliptic equation
\begin{align*}
\sum^{N}_{i=1}\left(-\partial_{ii}\right)^{s}u=u^{p},
\end{align*}
where $N\geq2$, $0<s<1$ and $p>1$. The notation $\left(-\partial_{ii}\right)^{s}$, which represents the fractional second-order directional derivative, serves as an extension of the definition of $-\partial_{ii}$ within the framework of fractional operators. However, the actions 
\begin{align*}
\sum^{N}_{i=1}\left(-\partial_{ii}\right)^{s}\quad\text{on }-\partial_{ii}
\end{align*} 
that initial nonlocalization followed by summarization make it greatly different from the well-known fractional Laplacian operator $\left(-\Delta\right)^{s}$. We will subsequently provide a precise definition of the operator $\left(-\partial_{ii}\right)^{s}$ and we will see that $-\Delta$ is the limiting operator for both $\left(-\Delta\right)^{s}$ and $\sum^{N}_{i=1}\left(-\partial_{ii}\right)^{s}$ as $s\rightarrow1^{-}$. 

To begin the study of the model, let us recall the long-time standing Lane--Emden equation
\begin{align}\label{E1-1}
-\Delta u=u^{p}\quad\text{in }\R^{N}.
\end{align}
We emphasize that the classification of positive solutions of \eqref{E1-1} depends on the selection of the exponent $p$. The thresholds for all $p>1$, which devide $p$ into distinct regimes, are categorized by the so-called 
\begin{align*}
\text{Serrin exponent}:\left\lbrace 
\begin{aligned}
&\frac{N}{N-2},&&N>2,\\
&+\infty,&&N=2.
\end{aligned}
\right.
\end{align*}

\begin{align*}
\text{Sobolev exponent}:\left\lbrace 
\begin{aligned}
&\frac{N+2}{N-2},&&N>2,\\
&+\infty,&&N=2.
\end{aligned}
\right.
\end{align*}

\begin{align*}
\text{Joseph-Lundgren exponent}:\left\lbrace 
\begin{aligned}
&1+\frac{4}{N^{2}-12N+20},&&N>10,\\
&+\infty,&&N\leq10
\end{aligned}
\right.
\end{align*}
and so on due to the flavor of references. The optimality of $1<p\leq\frac{N}{N-2}$ for the nonexistence of positive supersolutions was established by Gidas \cite{G}. Gidas--Spruck \cite{GS1,GS2} expanded this range of nonexistence for positive solutions to $1<p<\frac{N+2}{N-2}$. The Liouville result in subcritical case is considered optimal, as the uniqueness of solutions in the form established by Aubin \cite{A} and Talenti \cite{T} was demonstrated by Gidas--Ni--Nirenberg \cite{GNN2}, Caffarelli--Gidas--Spruck \cite{CGS}, Chen--Li \cite{CL} etc. in the critical case $p=\frac{N+2}{N-2}$. For the supercritical case $p>\frac{N+2}{N-2}$, Gui--Ni--Wang \cite{GNW} and Wang \cite{W} proved that \eqref{E1-1} possesses an unique positive radial solution under prescribed initial conditions. It is noteworthy that all of these radial solutions are unstable (particularly with infinite Morse index) if and only if $\frac{N+2}{N-2}<p<1+\frac{4}{N^{2}-12N+20}$, in which region Farina \cite{F} gave a sharp nonexistence result for any solution with a finite Morse index.

Motivated by the seminal works referenced for \eqref{E1-1}, there have been numerous extensions that involve a variety of quasilinear operators \cite{AS,BCN,GMS}, nonlocal operators \cite{CLL,DDW} and degenerate elliptic operators \cite {BC3,BGL}. These studies address the existence and nonexistence of solutions within the whole space as well as half spaces, exterior domains or bounded domains, see also \cite{BCC,BM,CL1,L} and reference therein. For further insights, We recommend the introduction of \cite{BDG} as a concise survey for interested readers.

Let us state our operator in studying the physical phenomenon, biological models and mathematical finance, particularly in relation to the integro-differential equation
\begin{align}\label{E1-2}
-Lu=u^{p}
\end{align}
which arose as a model for real-world phenomenon that involves L\'{e}vy processes. The operator $L$ is defined by
\begin{align*}
Lu\left(x\right)
=\frac{C_{s}}{2}\int_{\mathbb{S}^{N-1}}\int_{\R}\frac{u\left(x+t\theta\right)+u\left(x-t\theta\right)-2u\left(x\right)}{\left|t\right|^{1+2s}}dtd\mu, 
\end{align*}
where 
\begin{align*}
C_{s}=\left(\int_{\R}\frac{1-\cos x}{\left|x\right|^{1+2s}}dx\right)^{-1}
\end{align*}
and the spectral measure $\mu=\mu\left(\theta\right)$ defined on $\mathbb{S}^{N-1}$ satisfies the nondegenerate ellipticity conditions
\begin{align*}
0<\lambda\leq\inf_{\nu\in\mathbb{S}^{N-1}}\int_{\mathbb{S}^{N-1}}\left|\nu\cdot\theta\right|^{2s}d\mu\quad\text{and}\quad\int_{\mathbb{S}^{N-1}}d\mu\leq\Lambda<\infty.
\end{align*}
The operator $L$ is known as a class of the infinitesimal generators of the symmetric stable L\'{e}vy processes $X_{t}$. A simplest case of the spectral measure $\mu$ is $d\mu=Cd\theta$, in which case $L$ turns out to be $-L=C\left(-\Delta\right)^{s}$, where 
\begin{align*}
\left(-\Delta\right)^{s}u\left(x\right)
=C\left(N,s\right)\int_{\R^{N}}\frac{u\left(x+y\right)+u\left(x-y\right)-2u\left(x\right)}{\left|y\right|^{N+2s}}dy.
\end{align*}
There exists a lot of literature dedicated to the operator \((- \Delta)^{s}\). Notable works include the extension definition of fractional Laplacian by Caffarelli--Silvestre \cite{CS}, the guidance books of Bucur--Valdinoci \cite{BV} and Di Nezza--Palatucci--Valdinoci \cite{NPV}, among others referenced therein. As a specific instance of \eqref{E1-2}, equation \eqref{E1-1} is generalized to be
\begin{align}\label{E1-3}
\left(-\Delta\right)^{s}u=u^{p}\quad\text{in }\R^{N}.
\end{align}
We refer to Felmer--Quaas \cite{FQ} for the nonexistence of positive supersolutions of \eqref{E1-3} in the optimal range $1<p\leq\frac{N}{N-2s}$. Chen--Li--Ou\cite{CLO} and Chen--Li--Li \cite{CLL} have developed the moving plane method to classify all the positive solutions of \eqref{E1-3} when $1<p\leq\frac{N+2s}{N-2s}$. 

Recently, a more general absolutely continuous measure $\mu$ with $d\mu=a\left(\theta\right)d\theta$ was considered by Birindelli--Du--Galise \cite{BDG}, where $a\left(\theta\right)\in L^{\infty}\left(\mathbb{S}^{N-1}\right)$ is nonnegative and even function such that $L$ can be rewritten by
\begin{align*}
L_{a}u\left(x\right)
=\int_{\R^{N}}\frac{u\left(x+y\right)+u\left(x-y\right)-2u\left(x\right)}{\left|y\right|^{N+2s}}a\left(\frac{y}{\left|y\right|}\right)dy.
\end{align*}
Under the assumption that $a\left(\theta\right)$ is bounded above and not identically zero, they established the nonexistence of positive classical solutions of
\begin{align}\label{E1-4}
-L_{a}u\geq u^{p}
\end{align}
in $\R^{N}$ with the subcritical region $1<p\leq\frac{N}{N-2s}$ as well as in $\R_{+}^{N}$ along with the optimal exponent $1<p\leq\frac{N+s}{N-s}$.

In the present paper, we study another important non-absolutely continuous measure $\mu$: the L\'{e}vy processes $X_{t}$ is the vector of N-independent 1-dimensional symmetric stable processes 
\begin{align*}
X_{t}=\left(X^{1}_{t},...,X^{N}_{t}\right),
\end{align*}
in which case the spectral measure $\mu$ is the summation of Dirac measures
\begin{align*}
\mu=\sum^{N}_{i=1}\left(\delta_{e_{i}}+\delta_{-e_{i}}\right)
\end{align*}
on the unit orthogonal basis $\left\{e_{i}\right\}_{N}$. Correspondingly the infinitesimal generators $L$ of $X_{t}$ coincides with
\begin{align*}
\mathcal{I}u=\sum^{N}_{i=1}\mathcal{I}_{i}u,
\end{align*}
where
\begin{align*}
\mathcal{I}_{i}u\left(x\right)
:=\frac{C_{s}}{2}\int_{\R}\frac{u\left(x+te_{i}\right)+u\left(x-te_{i}\right)-2u\left(x\right)}{\left|t\right|^{1+2s}}dt.
\end{align*}
In this way, like the fractional Laplacian $\left(-\Delta\right)^{s}$, the operator $-\mathcal{I}$ can also be classified as a fractional Laplacian, originating from a different approach to nonlocalized extensions. Let us denote by
\begin{align*}
\left(-\partial_{ii}\right)^{s}u:=-\mathcal{I}_{i}u,
\end{align*}
it follows that for any $u\in C^{2}_{0}\left(\R^N\right)$, we have
\begin{align*}
\lim_{s\rightarrow1^{-}}\sum^{N}_{i=1}\left(-\partial_{ii}\right)^{s}u=\sum^{N}_{i=1}\lim_{s\rightarrow1^{-}}-\mathcal{I}_{i}u=\sum^{N}_{i=1}-\partial_{ii}u=-\Delta u.
\end{align*}
Consequently, the operator $-\Delta$ is one of the limiting forms of $-L$ as $s\rightarrow1^{-}$, not only in the context of $\left(-\Delta\right)^{s}$ but also in relation to $-\mathcal{I}$. 

One of the distinctions between $\left(-\Delta\right)^{s}$ and $-\mathcal{I}$ can be explained by the L\'{e}vy processes
with different jumps. The operator $\left(-\Delta\right)^{s}$ is $N$-dimensional spherically symmetric $2s$-stable operator whose L\'{e}vy processes jump in random directions over the unit sphere. This behavior ensures that the diffusion process encompasses the entirety of $N$-dimensional Lebesgue measure in the whole space. In contrast, the operator $-\mathcal{I}$ is composed of N-independent 1-dimensional symmetric $2s$-stable operators. Each process $X^{i}_{t}$ is restricted to jumping along the positive or negative coordinate axes randomly, resulting in a Lévy measure that is concentrated along the coordinate axes.

The regularity of solutions to the problem $-\mathcal{I}u=f$ has been established by Bass--Chen \cite{BC1,BC2}, Kassmann--Schwab \cite{KS}, Ros-Oton--Serra \cite{RS1,RS2}, Fern\'{a}ndez-Real--Ros-Oton \cite{FR} and reference therein. 

\subsection{Problem in the whole space}
We first study solutions of the following inequality
\begin{align}\label{E1-5}
-\mathcal{I}u\geq u^{p}\quad\text{in }\R^{N}.
\end{align}
The operator $\mathcal{I}$ is well-defined by assuming the usual condition that $u\in C^{2}\left(\R^{N}\right)\cap\mathcal{L}_{s}$, where
\begin{align*}
\mathcal{L}_{s}=\left\{u\in L^{1}_{loc}\left(\R^{N}\right):\limsup_{\left|x\right|\rightarrow+\infty}\frac{\left|u\right|}{\left|x\right|^{2s-\delta_{1}}}<+\infty\text{ for some }\delta_{1}\in(0,2s]\right\}.
\end{align*}
We say $u$ is a distributional solution of \eqref{E1-5} if for any nonnegative $\varphi\in C^{\infty}_{0}\left(\R^{N}\right)$, we have
\begin{align}\label{E1-6}
\int_{\R^{N}}u\left(-\mathcal{I}\varphi\right)dx\geq\int_{\R^{N}}u^{p}\varphi dx.
\end{align}
However, the classical solution $u\in C^{2}\left(\R^{N}\right)\cap\mathcal{L}_{s}$ cannot be directly extended to be the distributional solution in $\mathcal{L}_{s}$. The integration by parts formula for functions in $C^{2}\left(\R^{N}\right)\cap\mathcal{L}_{s}$ is satisfied for a class of operators with the kernel $\left|\cdot\right|^{-N-2s}$, whereas, it is not immediately obvious that such integration by parts can be performed for the operator $\mathcal{I}$, which is associated with the kernel $\left|\cdot\right|^{-1-2s}$, when considering two smooth functions in $\mathcal{L}_{s}$.

The left integral of \eqref{E1-6} may not be well-posed unless higher regularity conditions 
\begin{align*}
\int_{\R^{N}}\frac{u\left(x\right)}{1+\left|x\right|^{1+2s}}dx<+\infty
\end{align*}
are imposed. This would request that we substitute $\mathcal{L}_{s}$ with its subspace
\begin{align*}
\overline{\mathcal{L}_{s}}=\left\{u\in L^{1}_{loc}\left(\R^{N}\right):\limsup_{\left|x\right|\rightarrow+\infty}\frac{\left|u\right|}{\left|x\right|^{2s-N+1-\delta_{2}}}<+\infty\text{ for some }\delta_{2}>0\right\}.
\end{align*}
Now that the solution $u\in C^{2}\left(\R^{N}\right)\cap\overline{\mathcal{L}_{s}}$ of \eqref{E1-5} is also a distributional solution (See Corollary \ref{C2-9}) and the nonexistence of positive solutions can be established. 

\begin{thm}\label{T1-1}
Let $0<s<1$ and $1<p\leq\frac{N}{N-2s}$. Assume $u\in C^{2}\left(\R^{N}\right)\cap\overline{\mathcal{L}_{s}}$ is a nonnegative solution of \eqref{E1-5}, then $u\equiv0$.
\end{thm}

To prove the main results in this theorem, we will adopt the idea of Birindelli--Du--Galise \cite{BDG}, where they proved the nonexistence results for the positive classical solution of \eqref{E1-4} in the whole space and the half space. Motivated by the works in \cite{BCN,BDG}, we aim to establish a fundamental inequality involving the operator $\mathcal{I}$ and the compactly supported test functions in $\R^{N}$. Then, by using a scaled test-function argument, we are able to derive a Liouville result on and below the Serrin exponent $\frac{N}{N-2s}$, consistent with findings related to both local and nonlocal Laplacian operators. 

The validity of the proof for Theorem \ref{T1-1} is based on the applicability of the integration by parts formula, suggesting that the assumption $\overline{\mathcal{L}_{s}}$ is essential. Nevertheless, we relax the assumption from $\overline{\mathcal{L}_{s}}$ to $\mathcal{L}_{s}$, resulting in a more robust version of Theorem \ref{T1-1}. We address the aforementioned challenge by presenting an alternative proof.
 
\begin{thm}\label{T1-2}
Let $0<s<1$ and $1<p\leq\frac{N}{N-2s}$. Assume $u\in C^{2}\left(\R^{N}\right)\cap\mathcal{L}_{s}$ is a nonnegative solution of \eqref{E1-5}, then $u\equiv0$.
\end{thm}

We now focus on the positive solutions of the equation
\begin{align}\label{E1-7}
-\mathcal{I}u=u^{p}\quad\text{in }\R^{N}.
\end{align}
Obviously, the solutions of \eqref{E1-7} are invariant under the scaling 
\begin{align*}
\lambda^{\frac{2s}{p-1}}u\left(\lambda x+x_{0}\right),\quad\left(x_{0},\lambda\right)\in\R^{N}\times\R_{+}.
\end{align*}

As noted previously, the classification results pertaining to the positive solutions of \eqref{E1-1} and \eqref{E1-3} have been extensively investigated in the case $1<p\leq\frac{N+2s}{N-2s}$. Under the Kelvin transform
\begin{align}\label{E1-8}
u_{x_{0}}\left(x\right)=\frac{1}{\left|x-x_{0}\right|^{N-2s}}u\left(\frac{x-x_{0}}{\left|x-x_{0}\right|^{2}}+x_{0}\right),\quad x\in\R^{N}\backslash\left\{x_{0}\right\}
\end{align}
at some $x_{0}\in\R^{N}$, the function $u_{x_{0}}$ satisfies 
\begin{align*}
\left(-\Delta\right)^{s}u_{x_{0}}\left(x\right)=\frac{1}{\left|x-x_{0}\right|^{N+2s-p\left(N-2s\right)}}u_{x_{0}}^{p}\left(x\right),\quad x\in\R^{N}\backslash\left\{x_{0}\right\}
\end{align*}
and its equivalent integral equation
\begin{align*}
u_{x_{0}}\left(x\right)=\int_{\R^{N}}\frac{u_{x_{0}}^{p}\left(y\right)}{\left|x_{0}-y\right|^{N+2s-p\left(N-2s\right)}\left|x-y\right|^{N-2s}}dy,\quad x\in\R^{N}\backslash\left\{x_{0}\right\}.
\end{align*}
The moving plane method indicates that $u_{x_{0}}$ is radially symmetric, which leads to the nonexistence of $u$ provided $1<p<\frac{N+2s}{N-2s}$ and moreover $u$ enjoys the conformal invariance to possess unique form if $p=\frac{N+2s}{N-2s}$ \cite{CL,CLO}. 

However, the Kelvin transform is no longer suitably applicable to the solutions of \eqref{E1-7}. This limitation can be attributed to the absence of rotational invariance of the operator $\mathcal{I}$. Inspired by Gidas--Ni--Nirenberg \cite{GNN2}, we propose that $u$ satisfies the decay assumption
\begin{align}\label{E1-9}
u=o\left(\frac{1}{\left|x\right|^{\frac{2s}{p-1}}}\right)\quad\text{as }\left|x\right|\rightarrow+\infty.
\end{align}

For any unit vector $\nu\in\mathbb{S}^{N-1}$, we define $\left(N-1\right)$-dimensional hyperplane with unit normal vector $\nu$ that passes through a point $y\in\R^{N}$ by
\begin{align*}
P_{\nu,y}=\left\{x\in\R^{N}:\sum^{N}_{i=1}\nu_{i}\left(x_{i}-y_{i}\right)=0\right\}.
\end{align*}
The following point set will be used frequently:
\begin{align*}
\mathbf{S}_{k}:=\left\{1,...,k\right\},\quad k\in\N_{+}.
\end{align*}

By Theorem \ref{T1-2}, we may always assume $p>\frac{N}{N-2s}$. We shall prove that all the positive classical solutions of \eqref{E1-7} with the growth \eqref{E1-9} are symmetric. 
\begin{thm}\label{T1-3}
Let $0<s<1$ and $p>\frac{N}{N-2s}$. Assume $u\in C^{2}\left(\R^{N}\right)\cap\mathcal{L}_{s}$ is a positive solution of \eqref{E1-7} and $u$ satisfies \eqref{E1-9}. For any $i\in\mathbf{S}_{N}$,
\begin{itemize}
\item[(1)] $u$ is symmetric with respect to the plane
\begin{align*}
P_{e_{i},\widetilde{x}}=\left\{x\in\R^{N}:x_{i}=\widetilde{x}_{i}\right\}
\end{align*}
about some $\widetilde{x}\in\R^{N}$ and $u$ decreases in $e_{i}$ direction. 
\item[(2)] $u$ is not radially symmetric about $\widetilde{x}$.
\end{itemize}
\end{thm}
We employ the moving plane method to prove Theorem \ref{T1-3}. As a counterpart to \eqref{E1-7}, we consider the integral equation
\begin{align}\label{E1-10}
u=u^{p}\ast G_{s}\quad\text{in }\R^{N}.
\end{align}
We will introduce the potential $G_{s}$ in Section \ref{S2-5}. We use the moving plane method in integral form to show the symmetry of the positive solutions of \eqref{E1-10}.

\begin{thm}\label{T1-4}
Let $s=\frac{1}{2}$ and $p>\frac{N}{N-2s}$. Assume $u\in L^{\frac{N\left(p-1\right)}{2s}}\left(\R^{N}\right)$ is a positive solution of \eqref{E1-10}. For any $i,j\in\mathbf{S}_{N}$ with $i<j$,
\begin{itemize}
\item[(1)] $u$ is symmetric with respect to the plane
\begin{align*}
P_{e_{i},\widetilde{x}}=\left\{x\in\R^{N}:x_{i}=\widetilde{x}_{i}\right\}
\end{align*}
about some $\widetilde{x}\in\R^{N}$ and $u$ decreases in $e_{i}$ direction. 
\item[(2)] $u$ is also symmetric with respect to the plane
\begin{align*}
P_{e_{i}\pm e_{j},\widetilde{x}}=\left\{x\in\R^{N}:\left(x_{i}-\widetilde{x}_{i}\right)\pm\left(x_{j}-\widetilde{x}_{j}\right)=0\right\}
\end{align*}
and $u$ decreases in $e_{i}\pm e_{j}$ direction. 
\item[(3)] $u$ is not radially symmetric about $\widetilde{x}$.
\end{itemize}
\end{thm}

We anticipate that Theorem \ref{T1-4} may encompass all the range $0<s<1$, we have only been able to show in Section \ref{S2-5} that the potential $G_{s}$ is positive and decreasing when $s=\frac{1}{2}$. 

The assumptions that $u\in C^{2}\left(\R^{N}\right)\cap\mathcal{L}_{s}$ is a positive solution of \eqref{E1-7} and $u$ satisfies \eqref{E1-9} imply that $u\in L^{\frac{N\left(p-1\right)}{2s}}\left(\R^{N}\right)$ satisfies \eqref{E1-10}, see in Section \ref{S2-6}. Consequently, we can derive the following corollary.
\begin{cor}
Let $0<s<1$ and $p>\frac{N}{N-2s}$. Assume $u\in C^{2}\left(\R^{N}\right)\cap\mathcal{L}_{s}$ is a positive solution of \eqref{E1-7} and $u$ satisfies \eqref{E1-9}. For any $i,j\in\mathbf{S}_{N}$ with $i<j$,
\begin{itemize}
\item[(1)] $u$ is symmetric with respect to the plane $P_{e_{i},\widetilde{x}}$ about some $\widetilde{x}\in\R^{N}$ and $u$ decreases in $e_{i}$ direction. 
\item[(2)] When $s=\frac{1}{2}$, $u$ is also symmetric with respect to the plane $P_{e_{i}\pm e_{j},\widetilde{x}}$ and $u$ decreases in $e_{i}\pm e_{j}$ direction. 
\item[(3)] $u$ is not radially symmetric about $\widetilde{x}$.
\end{itemize}

\begin{figure}[htbp]
\centering
\includegraphics[width=0.4\textwidth]{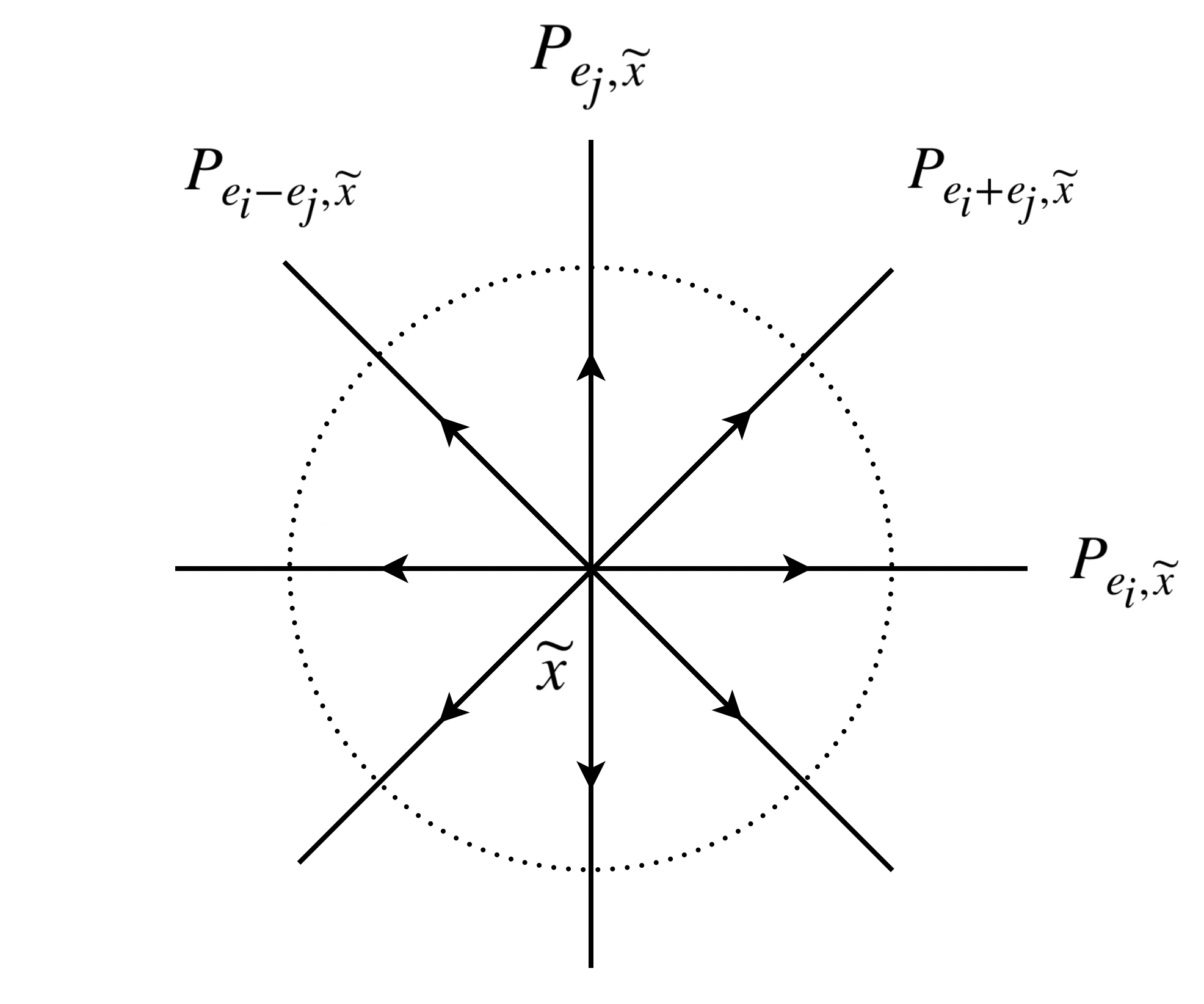}
\caption{Symmetry hyperplanes and decreasing directions when $s=\frac{1}{2}$.}
\end{figure}
\end{cor}

\subsection{Problem in the half space}
A well-known Liouville result proved by Berestycki--Capuzzo Dolcetta--Nirenberg \cite{BCN} in the half space
\begin{align*}
\R^{N}_{+}=\left\{x\in\R^{N}:x_{N}>0\right\}
\end{align*}
states that any nonnegative classical solution of 
\begin{align*}
-\Delta u\geq u^{p}\quad\text{in }\R_{+}^{N}
\end{align*} 
is trivial solution if and only if $1<p\leq\frac{N+1}{N-1}$. It is noteworthy that the extension result involving $\left(-\Delta\right)^{s}$ can be found in the result by Nornberg--dos Prazeres--Quaas \cite{NPQ}. 

We consider the nonnegative solutions $u\in C^{2}\left(\R_{+}^{N}\right)\cap\overline{\mathcal{L}_{s}}$ of the problem
\begin{align}\label{E1-11}
\left\lbrace 
\begin{aligned}
-\mathcal{I}u&\geq u^{p}&&\text{in }\R_{+}^{N},\\
u&\geq0&&\text{in }\R^N\backslash\R_{+}^{N}.
\end{aligned}
\right.
\end{align}
We establish a Liouville result as follows.
\begin{thm}\label{T1-6}
Let $0<s<1$ and $1<p\leq\frac{N+s}{N-s}$. Assume $u\in C^{2}\left(\R_{+}^{N}\right)\cap\overline{\mathcal{L}_{s}}$ is a nonnegative solution of \eqref{E1-11}, then $u\equiv0$.
\end{thm}

The proof of Theorem \ref{T1-6} is motivated by the arguments of Theorem \ref{T1-1}, where differently the non-smooth test function is much more involved. In particular, the diffusion of $\mathcal{I}$ is only along the coordinate axes directions, hence similar to the approach taken by Birindelli--Du--Galise \cite{BDG}, which addresses the operator $L_{a}$ with the non-vanishing diffusion along a conical cone, the idea of moving balls is also required. 

The nonexistence result of the positive classical solutions of 
\begin{align*}
-\Delta u=u^{p}\quad\text{in }\R_{+}^{N}
\end{align*}
for $1<p\leq\frac{N+2}{N-2}$ was proved by Gidas--Spruck \cite{GS1,GS2}. Chen--Fang--Yang \cite{CFY} and Chen--Li--Li \cite{CLL} extended the nonexistence result to $1<p\leq\frac{N+2s}{N-2s}$ for the fractional problem
\begin{align*}
\left\lbrace 
\begin{aligned}
\left(-\Delta\right)^{s}u&=u^{p}&&\text{in }\R_{+}^{N},\\
u&=0&&\text{in }\R^N\backslash\R_{+}^{N}.
\end{aligned}
\right.
\end{align*}
By applying the Kelvin transform \eqref{E1-8} at some $x_{0}\in\partial\R_{+}^{N}$, the positive function $u_{x_{0}}$ satisfies 
\begin{align*}
\left(-\Delta\right)^{s}u_{x_{0}}\left(x\right)=\frac{1}{\left|x-x_{0}\right|^{N+2s-p\left(N-2s\right)}}u_{x_{0}}^{p}\left(x\right),\quad x\in\R_{+}^{N}
\end{align*}
and also the integral equation involving the Green function $G_{+}$ of $\left(-\Delta\right)^{s}$ in $\R_{+}^{N}$, namely
\begin{align*}
u_{x_{0}}\left(x\right)=\int_{\R_{+}^{N}}G_{+}\left(x,y\right)\frac{u_{x_{0}}^{p}\left(y\right)}{\left|x_{0}-y\right|^{N+2s-p\left(N-2s\right)}}dy,\quad x\in\R_{+}^{N}.
\end{align*}
In this way a direct or integral form moving plane method works. 

We aim to establish Liouville type results to the nonlocal problem with Dirichlet boundary condition
\begin{align}\label{E1-12}
\left\lbrace 
\begin{aligned}
-\mathcal{I}u&=u^{p}&&\text{in }\R_{+}^{N},\\
u&=0&&\text{in }\R^N\backslash\R_{+}^{N}.
\end{aligned}
\right.
\end{align}
Since there is no Kelvin transform for the problem \eqref{E1-12}, the assumption \eqref{E1-9} is necessary to guarantee the appropriate decay of $u$ at infinity.
\begin{thm}\label{T1-7}
Let $0<s<1$ and $p>1$. Assume $u\in C^{2}\left(\R_{+}^{N}\right)\cap\mathcal{L}_{s}$ is a nonnegative solution of \eqref{E1-12} and $u$ satisfies \eqref{E1-9}, then $u\equiv0$.
\end{thm}

\subsection{Problem in the unit ball}
Finally, we establish a nonlocal version of Gidas--Ni--Nirenberg \cite{GNN1} problem
\begin{align}\label{E1-13}
\left\lbrace 
\begin{aligned}
-\mathcal{I}u&=f\left(u\right)&&\text{in }B_{1},\\
u&>0&&\text{in }B_{1},\\
u&=0&&\text{in }\R^{N}\backslash B_{1}.
\end{aligned}
\right.
\end{align}

\begin{thm}\label{T1-8}
Let $0<s<1$ and $f$ is a Lipschitz continuous function. Assume $u\in C^{2}\left(B_{1}\right)\cap\mathcal{L}_{s}$ is a solution of \eqref{E1-13}. For any $i\in\mathbf{S}_{N}$, $u$ is symmetric with respect to the plane $P_{e_{i},0}$ and $u$ decreases in $e_{i}$ direction. 
\end{thm}

The paper is organized as follows. 
\begin{itemize}
\item In Section \ref{S2} we present some knowledge of the fractional operator $\left(-\partial_{ii}\right)^{s}$. 
\item In Section \ref{S3} we deal with the problems in the whole space involving Theorem \ref{T1-1}-\ref{T1-4}. \item In Section \ref{S4} we prove Theorem \ref{T1-6} and Theorem \ref{T1-7} in the half space. 
\item In Section \ref{S5} we explore symmetry in the unit ball by proving Theorem \ref{T1-8}.
\end{itemize}

\section{Preliminary}\label{S2}
In this section, we present a substantial amount of preliminary knowledge of the fractional operator $\left(-\partial_{ii}\right)^{s}$ relevant to the paper. We denote $C_{\ast}$ or $C\left({\ast}\right)$ as positive constants that are contingent upon the variable $\ast$. We define $B_{R}\left(x\right)$ as the ball centered at the point $x$ with a radius of $R$ and we use $B_{R}$ to refer to this ball when the center is located at the origin.

It would be appropriate to occasionally omit the constant coefficient in the definition of $I_{i}$ such that
\begin{align*}
\mathcal{I}_{i}u\left(x\right)
=\int_{\R}\frac{u\left(x+te_{i}\right)+u\left(x-te_{i}\right)-2u\left(x\right)}{\left|t\right|^{1+2s}}dt.
\end{align*}
Furthermore, in order to facilitate the calculations, we will sometimes employ an equivalent formulation of $\mathcal{I}_{i}$, namely
\begin{align*}
\mathcal{I}_{i}u\left(x\right)=\text{P.V.}\int_{\R}\frac{u\left(x+te_{i}\right)-u\left(x\right)}{\left|t\right|^{1+2s}}dt.
\end{align*}
We first give some properties regarding the operator $\mathcal{I}$, which will be used in our proofs.
\subsection{Basic properties}
\begin{lem}\label{L2-1}
Let $i\in\mathbf{S}_{N}$ and $\varphi\in C^{2}_{0}\left(\R^{N}\right)$. For any $R>0$, we define 
\begin{align*}
\varphi_{R}\left(x\right)=\varphi\left(\frac{x}{R}\right).
\end{align*}
Then we have
\begin{align*}
\mathcal{I}_{i}\varphi_{R}\left(x\right)=R^{-2s}\mathcal{I}_{i}\varphi\left(\frac{x}{R}\right).
\end{align*}
\end{lem}
\begin{proof}
For any $R>0$ and $x\in\R^{N}$, let $t=Rr$, there holds
\begin{align*}
\mathcal{I}_{i}\varphi_{R}\left(x\right)
&=\int_{\R}\frac{\varphi\left(\frac{x+te_{i}}{R}\right)+\varphi\left(\frac{x-te_{i}}{R}\right)-2\varphi\left(\frac{x}{R}\right)}{\left|t\right|^{1+2s}}dt\\
&=R^{-2s}\int_{\R}\frac{\varphi\left(\frac{x}{R}+re_{i}\right)+\varphi\left(\frac{x}{R}-re_{i}\right)-2\varphi\left(\frac{x}{R}\right)}{\left|r\right|^{1+2s}}dr\\
&=R^{-2s}\mathcal{I}_{i}\varphi\left(\frac{x}{R}\right).
\end{align*}
\end{proof}

\begin{lem}\label{L2-2}
Let $i\in\mathbf{S}_{N}$ and $m>1$. For any nonnegative function $\varphi\in C^{2}_{0}\left(\R^{N}\right)$, we have
\begin{align*} 
\mathcal{I}_{i}\varphi^{m}\geq m\varphi^{m-1}\mathcal{I}_{i}\varphi.
\end{align*}
\end{lem}
\begin{proof}
By the definition of $\mathcal{I}_{i}$, we know that for any $x\in\R^{N}$, 
\begin{align*}
\varphi^{m-1}\left(x\right)\mathcal{I}_{i}\varphi\left(x\right)
=\int_{\R}\frac{\varphi^{m-1}\left(x\right)\varphi\left(x+te_{i}\right)+\varphi^{m-1}\left(x\right)\varphi\left(x-te_{i}\right)-2\varphi^{m}\left(x\right)}{\left|t\right|^{1+2s}}dt.
\end{align*}
By Young's inequality, we have
\begin{align*}
\varphi^{m-1}\left(x\right)\varphi\left(x\pm te_{i}\right)\leq\frac{m-1}{m}\varphi^{m}\left(x\right)+
\frac{1}{m}\varphi^{m}\left(x\pm te_{i}\right).
\end{align*}
Thus we get
\begin{align*}
\varphi^{m-1}\left(x\right)\mathcal{I}_{i}\varphi\left(x\right)
\leq\frac{1}{m}\int_{\R}\frac{\varphi^{m}\left(x+te_{i}\right)+\varphi^{m}\left(x-te_{i}\right)-2\varphi^{m}\left(x\right)}{\left|t\right|^{1+2s}}dt=\frac{1}{m}\mathcal{I}_{i}\varphi^{m}\left(x\right).
\end{align*}
\end{proof}

It is known by \cite{RS1,RS2} that the function $\left(x_{N}\right)_{+}^{s}$ is a solution of 
\begin{align*}
\left\lbrace 
\begin{aligned}
-Lu&=0&&\text{in }\R_{+}^{N},\\
u&=0&&\text{in }\R^N\backslash\R_{+}^{N}.
\end{aligned}
\right.
\end{align*}
The following lemma can be readily derived.

\begin{lem}\label{L2-3}
Let $0<\alpha<2s$ and $\omega_{\alpha}\left(x\right)=\left(x_{N}\right)_{+}^{\alpha}$. For any $x\in\R^{N}_{+}$, we have
\begin{align*}
\mathcal{I}\omega_{\alpha}\left(x\right)=C_{\alpha}x_{N}^{\alpha-2s},
\end{align*}
where
\begin{align*}
C_{\alpha}
\left\lbrace 
\begin{aligned}
&<0,&&0<\alpha<s,\\
&=0,&&\alpha=s,\\
&>0,&&s<\alpha<2s.
\end{aligned}
\right.
\end{align*}
\end{lem}

\begin{proof}
For any $x\in\R^{N}_{+}$, when $i\in\mathbf{S}_{N-1}$,
\begin{align*}
\mathcal{I}_{i}\omega_{\alpha}\left(x\right)
&=\int_{\R}\frac{x_{N}^{\alpha}+x_{N}^{\alpha}-2x_{N}^{\alpha}}{\left|t\right|^{1+2s}}dt=0.
\end{align*}
When $i=N$, let $t=x_{N}r$, we get
\begin{align*}
\mathcal{I}_{N}\omega_{\alpha}\left(x\right)
&=\int_{\R}\frac{\left(x_{N}+t\right)_{+}^{\alpha}+\left(x_{N}-t\right)_{+}^{\alpha}-2x_{N}^{\alpha}}{\left|t\right|^{1+2s}}dt\\
&=x_{N}^{\alpha-2s}\int_{\R}\frac{\left(1+r\right)_{+}^{\alpha}+\left(1-r\right)_{+}^{\alpha}-2}{r^{1+2s}}dr.
\end{align*}
Therefore
\begin{align*}
\mathcal{I}\omega_{\alpha}\left(x\right)=x_{N}^{\alpha-2s}\int_{\R}\frac{\left(1+r\right)_{+}^{\alpha}+\left(1-r\right)_{+}^{\alpha}-2}{r^{1+2s}}dr.
\end{align*}
The desire result follows by the fact
\begin{align*}
C_{\alpha}=:\int_{\R}\frac{\left(1+r\right)_{+}^{\alpha}+\left(1-r\right)_{+}^{\alpha}-2}{r^{1+2s}}dr
\left\lbrace 
\begin{aligned}
&<0,&&0<\alpha<s,\\
&=0,&&\alpha=s,\\
&>0,&&s<\alpha<2s.
\end{aligned}
\right.
\end{align*}
\end{proof}

\begin{lem}\label{L2-4}
Let $f,g,h\in C^{2}\left(\R_{+}^{N}\right)\cap\mathcal{L}_{s}$ and $f=gh$. For any $x\in\R_{+}^{N}$, we have
\begin{align*}
\mathcal{I}f
=g\mathcal{I}h+h\mathcal{I}g+I\left[g,h\right],
\end{align*}
where
\begin{align*}
I\left[g,h\right]=\sum_{i=1}^{N}I_{i}\left[g,h\right]
\end{align*}
with
\begin{align*}
I_{i}\left[g,h\right]\left(x\right)=&\int_{\R}\frac{\left[g\left(x+te_{i}\right)-g\left(x\right)\right]\left[h\left(x+te_{i}\right)-h\left(x\right)\right]}{\left|t\right|^{1+2s}}\\
&+\frac{\left[g\left(x-te_{i}\right)-g\left(x\right)\right]\left[h\left(x-te_{i}\right)-h\left(x\right)\right]}{\left|t\right|^{1+2s}}dt.
\end{align*}
\end{lem}
\begin{proof}
For any $i\in\mathbf{S}_{N}$ and $\left(x,t\right)\in\R_{+}^{N}\times\R$,
\begin{align*}
&g\left(x+te_{i}\right)h\left(x+te_{i}\right)+g\left(x-te_{i}\right)h\left(x-te_{i}\right)-2g\left(x\right)h\left(x\right)\\
&=\left[g\left(x+te_{i}\right)+g\left(x-te_{i}\right)-2g\left(x\right)\right]h\left(x\right)+\left[h\left(x+te_{i}\right)+h\left(x-te_{i}\right)-2h\left(x\right)\right]g\left(x\right)\\
&\quad+\left[g\left(x+te_{i}\right)-g\left(x\right)\right]\left[h\left(x+te_{i}\right)-h\left(x\right)\right]+\left[g\left(x-te_{i}\right)-g\left(x\right)\right]\left[h\left(x-te_{i}\right)-h\left(x\right)\right].
\end{align*}
The result is derived by substituting the above identity into the definition of $\mathcal{I}_{i}f$.
\end{proof}

We present the continuity property of $I$ as outlined in Lemma \ref{L2-4}.

\begin{lem}\label{L2-5}
Let $\left(2s-1\right)_+<\alpha<\min\left\{1,2s\right\}$, $\varphi\in C^{1}_{0}\left(\R^{N}\right)$ and $\omega_{\alpha}\left(x\right)=\left(x_{N}\right)_{+}^{\alpha}$. For any $x\in\overline{\R_{+}^{N}}$, we have
\begin{align*}
I\left[\omega_{\alpha},\varphi\right]
=I_{N}\left[\omega_{\alpha},\varphi\right].
\end{align*}
Moreover, for any $\left\{x_{n}\right\}\subset\R_{+}^{N}$ and $x_{n}\rightarrow x_{0}\in\overline{\R_{+}^{N}}$ as $n\rightarrow+\infty$, we have
\begin{align*}
\lim_{n\rightarrow+\infty}I_{N}\left[\omega_{\alpha},\varphi\right]\left(x_{n}\right)
=I_{N}\left[\omega_{\alpha},\varphi\right]\left(x_{0}\right).
\end{align*}
\end{lem}
\begin{proof}
For any $x\in\overline{\R_{+}^{N}}$, by the definition of $I_{i}$ in Lemma \ref{L2-4}, when $i\in\mathbf{S}_{N-1}$,
\begin{align*}
I_{i}\left[\omega_{\alpha},\varphi\right]\left(x\right)
=&\int_{\R}\frac{\left[x_{N}^{\alpha}-x_{N}^{\alpha}\right]\left[\varphi\left(x+te_{i}\right)-\varphi\left(x\right)\right]+\left[x_{N}^{\alpha}-x_{N}^{\alpha}\right]\left[\varphi\left(x-te_{i}\right)-\varphi\left(x\right)\right]}{\left|t\right|^{1+2s}}dt=0.
\end{align*}
Hence we have
\begin{align*}
I\left[\omega_{\alpha},\varphi\right]\left(x\right)=I_{N}\left[\omega_{\alpha},\varphi\right]\left(x\right),\quad\forall x\in\overline{\R_{+}^{N}}.
\end{align*}

For any convergent sequence $\left\{x_{n}\right\}\subset\R_{+}^{N}$, we define
\begin{align}\label{E2-1}
\begin{split}
I_{N}\left[\omega_{\alpha},\varphi\right]\left(x_{n}\right)
&=\int_{\R}\frac{\left[\left(\left(x_{n}\right)_{N}+t\right)_{+}^{\alpha}-\left(x_{n}\right)_{N}^{\alpha}\right]\left[\varphi\left(x_{n}+te_{N}\right)-\varphi\left(x_{n}\right)\right]}{\left|t\right|^{1+2s}}\\
&\quad+\frac{\left[\left(\left(x_{n}\right)_{N}-t\right)_{+}^{\alpha}-\left(x_{n}\right)_{N}^{\alpha}\right]\left[\varphi\left(x_{n}-te_{N}\right)-\varphi\left(x_{n}\right)\right]}{\left|t\right|^{1+2s}}dt\\
&=:\int_{\R}H\left(x_{n},t\right)dt.
\end{split}
\end{align}
Since $0<\alpha<1$, for any $t\in\R$,
\begin{align*}
\left|\left(\left(x_{n}\right)_{N}\pm t\right)_{+}^{\alpha}-\left(x_{n}\right)_{N}^{\alpha}\right|\leq \left|t\right|^{\alpha}.
\end{align*}
On the other hand,
\begin{align*}
\left|\varphi\left(x_{n}\pm te_{N}\right)-\varphi\left(x_{n}\right)\right|\leq C\left(\left\|\varphi\right\|_{C^{1}\left(\R^{N}\right)}\right)
\times
\left\lbrace 
\begin{aligned}
&\left|t\right|,&&\left|t\right|<1,\\
&1,&&\left|t\right|\geq1.
\end{aligned}
\right.
\end{align*}
Hence we get by \eqref{E2-1} that
\begin{align*}
\left|H\left(x_{n},t\right)\right|
\leq H\left(t\right)=C\left(\left\|\varphi\right\|_{C^{1}\left(\R^{N}\right)}\right)\times
\left\lbrace 
\begin{aligned}
&\frac{1}{\left|t\right|^{2s-\alpha}},&&\left|t\right|<1,\\
&\frac{1}{\left|t\right|^{1+2s-\alpha}},&&\left|t\right|\geq1.
\end{aligned}
\right.
\end{align*}
This inequality, together with the assumption $2s-1<\alpha<2s$, implies that $H\left(t\right)\in L^{1}\left(\R\right)$. Consequently, Lebesgue's dominated convergence theorem, along with \eqref{E2-1}, shows that 
\begin{align*}
\lim_{n\rightarrow+\infty}I_{N}\left[\omega_{\alpha},\varphi\right]\left(x_{n}\right)
=\lim_{n\rightarrow+\infty}\int_{\R}H\left(x_{n},t\right)dt
=\int_{\R}H\left(x_{0},t\right)dt
=I_{N}\left[\omega_{\alpha},\varphi\right]\left(x_{0}\right).
\end{align*}
\end{proof}

\subsection{Maximum principle }
For any $x\in\R^{N}$, we define a diffusion domain along the coordinate axis directions by
\begin{align*}
A\left(x\right):=\left\{y\in\R^{N}:y=x\pm te_{i},\forall t\in\R,\forall i\in\mathbf{S}_{N}\right\}.
\end{align*}
Let $\Omega\subseteq\R^{N}$ be an open domain, we define 
\begin{align*}
A\left(\Omega\right)=\bigcup_{x\in\Omega}A\left(x\right).
\end{align*} 

\begin{figure}[htbp]
\centering
\includegraphics[width=0.35\textwidth]{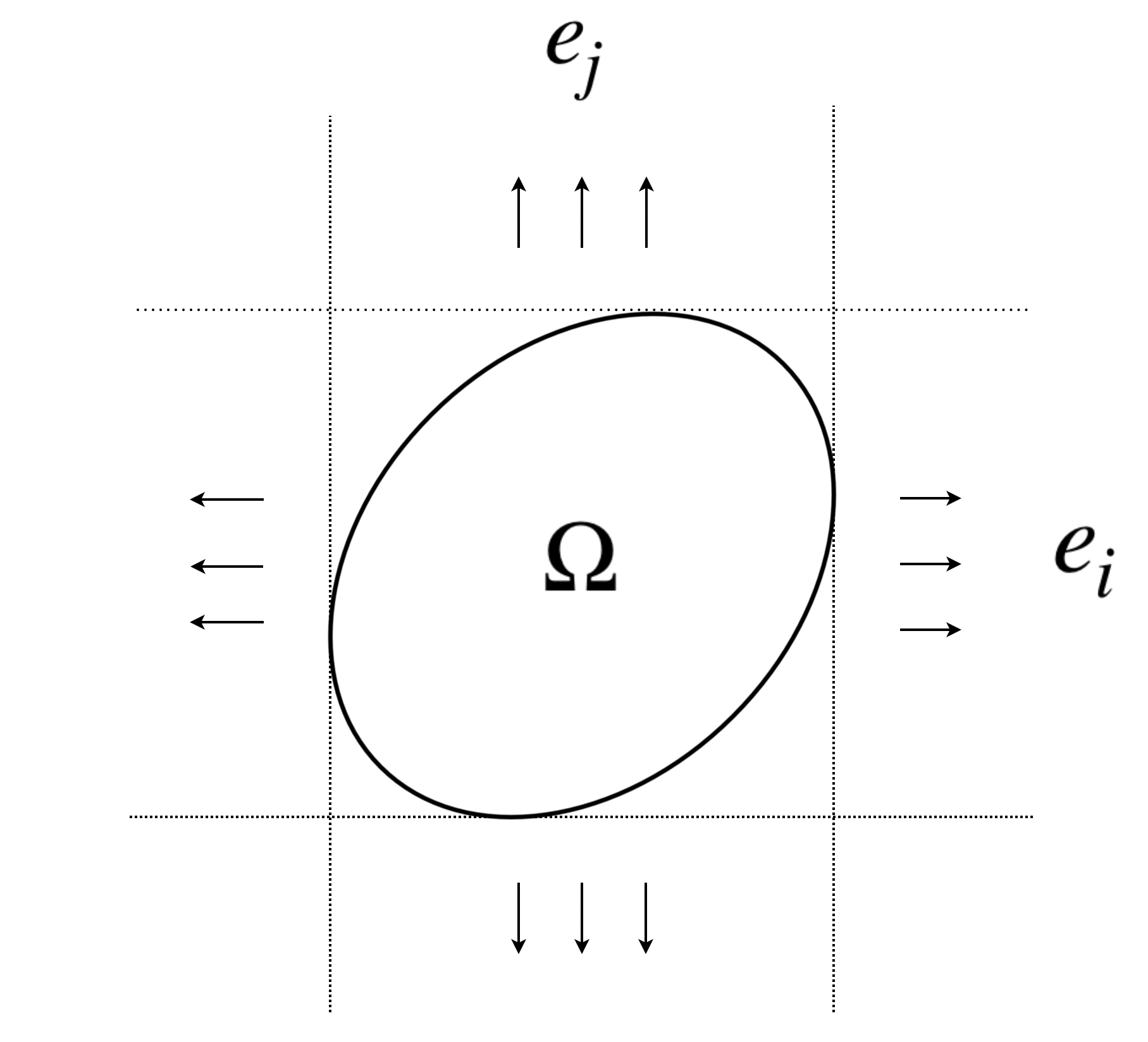}
\caption{Diffusion along the coordinate axis directions}
\end{figure}

The following lemma describes a maximum principle for $\mathcal{I}$. 
\begin{lem}\label{L2-6}
Let $\Omega\subseteq\R^{N}$ be an open domain. Assume $u\in C^{2}\left(\Omega\right)\cap\mathcal{L}_{s}$ is a solution of 
\begin{align}\label{E2-2}
\left\lbrace 
\begin{aligned}
-\mathcal{I}u\left(x\right)&\geq0,&&x\in\Omega,\\
u\left(x\right)&\geq0,&&x\in\R^{N}\backslash\Omega
\end{aligned}
\right.
\end{align}
and $u$ is lower semi-continuous on $\overline{\Omega}$. 
\begin{itemize}
\item[1.] If $\Omega$ is bounded, then $u\geq0$ in $\Omega$.
\item[2.] If $\Omega$ is unbounded and $\liminf_{\left|x\right|\rightarrow+\infty}u\geq0$, then $u\geq0$ in $\Omega$.
\end{itemize}
If moreover there exists $x_{0}$ such that $u\left(x_{0}\right)=0$ in $\Omega$, then $u=0$ a.e. in $A\left(\Omega\right)$.
\end{lem}
\begin{proof}
We assume that $u\geq0$ does not hold in $\Omega$, then by the lower semi-continuity of $u$ on $\overline{\Omega}$, whenever $\Omega$ is bounded or not, there exists a negative minimum point $z$ of $u$ in $\Omega$. Hence
\begin{align*}
-\mathcal{I}u\left(z\right)
&=-\sum_{i=1}^{N}\int_{\R}\frac{u\left(z+te_{i}\right)+u\left(z-te_{i}\right)-2u\left(z\right)}{\left|t\right|^{1+2s}}dt<0,
\end{align*}
which contradicts \eqref{E2-2}. Thus we have $u\geq0$ in $\Omega$. 

If there exists $x_{0}$ such that $u\left(x_{0}\right)=0$ in $\Omega$, then the fact $u\geq0$ in $\R^{N}$ implies that
\begin{align*}
0\leq-\mathcal{I}u\left(x_{0}\right)
&=-\sum_{i=1}^{N}\int_{\R}\frac{u\left(x_{0}+te_{i}\right)+u\left(x_{0}-te_{i}\right)}{\left|t\right|^{1+2s}}dt\leq0,
\end{align*}
that is $u=0$ a.e. in $A\left(x_{0}\right)$. We can substitute $x_{0}$ with any point in $A\left(x_{0}\right)\cap\Omega$, let us repeat this process, we ultimately conclude that $u=0$ a.e. in $A\left(\Omega\right)$.
\end{proof}

For $\lambda\in\mathbb{R}$, we denote
\begin{align*}
x^{\lambda}=\left(2\lambda-x_{1},...,x_{N}\right)
\end{align*}
and
\begin{align*}
\Sigma_{\lambda}=\left\{x\in\R^{N}:x_{1}<\lambda\right\}.
\end{align*}
Let us define the narrow domain $D$ which assumes
\begin{align*}
D\subset\Sigma_{\lambda}\backslash\Sigma_{\lambda-d}=\left\{x\in\R^{N}:\lambda-d\leq x_{1}<\lambda,d>0\right\}.
\end{align*}
We establish a maximum principle for the anti-symmetric functions in $D$.

\begin{lem}\label{L2-7}
Let $\lambda\in\R$, $d>0$ and $\psi>-\left(sd\right)^{-2s}$ in $D$. Assume $u\in C^{2}\left(D\right)\cap\mathcal{L}_{s}$ is lower semi-continuous on $\overline{D}$ and $u$ satisfies 
\begin{align}\label{E2-3}
\left\lbrace 
\begin{aligned}
-\mathcal{I}u\left(x\right)&+\psi\left(x\right)u\left(x\right)\geq0,&&x\in D,\\
u\left(x\right)&\geq0,&&x\in\Sigma_{\lambda}\backslash D,\\
u\left(x^{\lambda}\right)&=-u\left(x\right),&&x\in D.
\end{aligned}
\right.
\end{align}
We have
\begin{itemize}
\item[1.] If $D$ is bounded, then $u\geq0$ in $D$.
\item[2.] If $D$ is unbounded and $\liminf_{\left|x\right|\rightarrow+\infty}u\geq0$, then $u\geq0$ in $D$.
\end{itemize}
If moreover there exists $x_{0}$ such that $u\left(x_{0}\right)=0$ in $D$, then $u=0$ a.e. in $A\left(D\right)$.
\end{lem}
\begin{proof}
We assume $u\geq0$ in $D$ is not true, by the lower semi-continuity of $u$ in $\overline{D}$, whenever $D$ is bounded or not, there exists a negative minimum point $z$ of $u$ in $D$. Let us denote 
\begin{align*}
z_{t}=z+\left(t-z_{1}\right)e_{1}=\left(t,z_{2},...,z_{N}\right).
\end{align*}
When $i=1$, we have
\begin{align*}
\mathcal{I}_{1}u\left(z\right)
&=\text{P.V.}\int_{\R}\frac{u\left(z+te_{1}\right)-u\left(z\right)}{\left|t\right|^{1+2s}}dt\\
&=\text{P.V.}\int_{\R}\frac{u\left(z_{t}\right)-u\left(z\right)}{\left|t-z_{1}\right|^{1+2s}}dt\\
&=\text{P.V.}\int_{-\infty}^{\lambda}\frac{u\left(z_{t}\right)-u\left(z\right)}{\left|t-z_{1}\right|^{1+2s}}dt+\int_{\lambda}^{+\infty}\frac{u\left(z_{t}\right)-u\left(z\right)}{\left|t-z_{1}\right|^{1+2s}}dt.
\end{align*}
Moreover, we obtain
\begin{align*}
\text{P.V.}\int_{-\infty}^{\lambda}\frac{u\left(z_{t}\right)-u\left(z\right)}{\left|t-z_{1}\right|^{1+2s}}dt
\geq\int_{-\infty}^{\lambda}\frac{u\left(z_{t}\right)-u\left(z\right)}{\left|2\lambda-t-z_{1}\right|^{1+2s}}dt
\end{align*}
and
\begin{align*}
\int_{\lambda}^{+\infty}\frac{u\left(z_{t}\right)-u\left(z\right)}{\left|t-z_{1}\right|^{1+2s}}dt
=\int_{-\infty}^{\lambda}\frac{u\left(z_{t}^{\lambda}\right)-u\left(z\right)}{\left|2\lambda-t-z_{1}\right|^{1+2s}}dt
=\int_{-\infty}^{\lambda}\frac{-u\left(z_{t}\right)-u\left(z\right)}{\left|2\lambda-t-z_{1}\right|^{1+2s}}dt.
\end{align*}
Hence we conclude that
\begin{align*}
\mathcal{I}_{1}u\left(z\right)
&\geq\int_{-\infty}^{\lambda}\frac{u\left(z_{t}\right)-u\left(z\right)}{\left|2\lambda-t-z_{1}\right|^{1+2s}}dt+\int_{-\infty}^{\lambda}\frac{-u\left(z_{t}\right)-u\left(z\right)}{\left|2\lambda-t-z_{1}\right|^{1+2s}}dt\\
&=-2u\left(z\right)\int_{-\infty}^{\lambda}\frac{1}{\left|2\lambda-t-z_{1}\right|^{1+2s}}dt\\
&=-2u\left(z\right)\int^{+\infty}_{\lambda}\frac{1}{\left|t-z_{1}\right|^{1+2s}}dt.
\end{align*}
When $i=2,...,N$, there is $z+te_{i}\in D$, then
\begin{align*}
\mathcal{I}_{i}w_{\lambda}\left(z\right)
=\text{P.V.}\int_{\R}\frac{u\left(z+te_{i}\right)-u\left(z\right)}{\left|t\right|^{1+2s}}dt\geq0.
\end{align*}
Therefore
\begin{align}\label{E2-4}
-\mathcal{I}u\left(z\right)
\leq2u\left(z\right)\int^{+\infty}_{\lambda}\frac{1}{\left|t-z_{1}\right|^{1+2s}}dt.
\end{align}
Notice that for $z\in D$,
\begin{align*}
\int_{\lambda}^{+\infty}\frac{1}{\left|t-z_{1}\right|^{1+2s}}dt
\geq\int_{z_{1}+d}^{+\infty}\frac{1}{\left|t-z_{1}\right|^{1+2s}}dt
=\frac{1}{2sd^{2s}},
\end{align*}
thus we have
\begin{align*}
-\mathcal{I}u\left(z\right)
\leq\frac{u\left(z\right)}{sd^{2s}}.
\end{align*}
Consequently, since that $\psi>-\left(sd\right)^{-2s}$ in $D$, there holds
\begin{align*}
-\mathcal{I}u\left(z\right)+\psi\left(z\right)u\left(z\right)<\left(\frac{1}{sd^{2s}}+\psi\left(z\right)\right)u\left(z\right)<0,
\end{align*}
which contradicts \eqref{E2-3}. We get $u\geq0$ in $D$. The other portion of the proof is analogous to that of Lemma \ref{L2-6}.
\end{proof}

\subsection{Integration by parts formula}
We primarily focus on the integration by parts formula associated with the operator $\mathcal{I}$. As previously indicated in the introduction, the applicability of integration by parts formula is crucial for the analysis of the classical solutions to semilinear or fully nonlinear problems that involve the operator $\mathcal{I}$. A significant distinction arises from the 1-dimensional kernel $\left|t\right|^{-1-s}$, which exhibits divergence in its $N$-dimensional integral at infinity. Consequently, we can perform integration by parts in 1-dimensional contexts or for functions that have a more rapid decay rate.

\begin{lem}\label{L2-8}
Let $i\in\mathbf{S}_{N}$ and $\varphi\in C^{2}_{0}\left(\R^{N}\right)$. Then for any $u\in C^{2}\left(\R^{N}\right)\cap\overline{\mathcal{L}_{s}}$, we have
\begin{align*}
\int_{\R^{N}\times\R}\frac{\varphi\left(x+te_{i}\right)+\varphi\left(x-te_{i}\right)-2\varphi\left(x\right)}{\left|t\right|^{1+2s}}u\left(x\right)dxdt<+\infty.
\end{align*}
For any $u\in C^{2}\left(\R^{N}\right)\cap\mathcal{L}_{s}$, we have
\begin{align*}
\int_{\R^{N}\times\R}\frac{u\left(x+te_{i}\right)+u\left(x-te_{i}\right)-2u\left(x\right)}{\left|t\right|^{1+2s}}\varphi\left(x\right)dxdt<+\infty.
\end{align*}
\end{lem}
\begin{proof}
We assume supp $\varphi\subseteq B_{R}$ with $R>1$. We define two maps $F_{1},F_{2}:\R^{N}\times\left(\R\backslash\left\{0\right\}\right)\rightarrow\R$ by
\begin{align*}
F_{1}\left(x,t\right):=\frac{\varphi\left(x+te_{i}\right)+\varphi\left(x-te_{i}\right)-2\varphi\left(x\right)}{\left|t\right|^{1+2s}}u\left(x\right)
\end{align*}
and
\begin{align*}
F_{2}\left(x,t\right):=\frac{u\left(x+te_{i}\right)+u\left(x-te_{i}\right)-2u\left(x\right)}{\left|t\right|^{1+2s}}\varphi\left(x\right).
\end{align*}
We will prove that for any $u\in C^{2}\left(\R^{N}\right)\cap\overline{\mathcal{L}_{s}}$, we have
\begin{align}\label{E2-5}
\int_{\R^{N}}\left(\int_{\R}\left|F_{1}\right|dt\right)dx<+\infty,
\end{align}
and for any $u\in C^{2}\left(\R^{N}\right)\cap\mathcal{L}_{s}$, we have
\begin{align}\label{E2-6}
\int_{\R^{N}}\left(\int_{\R}\left|F_{2}\right|dt\right)dx<+\infty,
\end{align}
then Lemma \ref{L2-8} is derived from Tonelli's theorem.

Firstly, we prove \eqref{E2-5}. When $\left|x\right|\geq 2R$, we notice that if $\left|x\pm te_{i}\right|\leq R$, then
\begin{align*}
\left|t\right|\geq\left|x\right|-\left|x\pm te_{i}\right|\geq\frac{\left|x\right|}{2},
\end{align*}
hence we have
\begin{align*}
\int_{\R}\frac{\left|\varphi\left(x+te_{i}\right)+\varphi\left(x-te_{i}\right)-2\varphi\left(x\right)\right|}{\left|t\right|^{1+2s}}dt
&=2\int_{\left|x+te_{i}\right|\leq R}\frac{\left|\varphi\left(x+te_{i}\right)\right|}{\left|t\right|^{1+2s}}dt\\
&\leq \left(4^{1+s}\left\|\varphi\right\|_{L^{\infty}\left(\R^{N}\right)}\int_{\left|x_{i}+t\right|\leq R}dt\right)\frac{1}{\left|x\right|^{1+2s}}.
\end{align*}
When $\left|x\right|<2R$, there holds
\begin{align}\label{E2-7}
\int_{\R}\frac{\left|\varphi\left(x+te_{i}\right)+\varphi\left(x-te_{i}\right)-2\varphi\left(x\right)\right|}{\left|t\right|^{1+2s}}dt\leq C\left(s,R,\left\|\varphi\right\|_{C^{2}\left(\R^{N}\right)}\right).
\end{align}
We conclude that
\begin{align}\label{E2-8}
\begin{split}
&\int_{\R}\frac{\left|\varphi\left(x+te_{i}\right)+\varphi\left(x-te_{i}\right)-2\varphi\left(x\right)\right|}{\left|t\right|^{1+2s}}dt
\leq C\left(s,R,\left\|\varphi\right\|_{C^{2}\left(\R^{N}\right)}\right)\times
\left\lbrace 
\begin{aligned}
&1,&&\left|x\right|<2R,\\
&\frac{1}{\left|x\right|^{1+2s}},&&\left|x\right|\geq 2R.
\end{aligned}
\right.
\end{split}
\end{align}
Since that $u\in C^{2}\left(\R^{N}\right)\cap\overline{\mathcal{L}_{s}}$, we can infer that there exists $\beta_{1}=\beta_{1}\left(N,s,\delta_{2},R,u\right)>0$ such that
\begin{align}\label{E2-9}
\left|u\left(x\right)\right|\leq\beta_{1}\left|x\right|^{2s-N+1-\delta_{2}},\quad\forall \left|x\right|\geq 2R.
\end{align}
Therefore, by \eqref{E2-8} and \eqref{E2-9}, for any $u\in C^{2}\left(\R^{N}\right)\cap\overline{\mathcal{L}_{s}}$ we get
\begin{align*}
\int_{\R^{N}}\left(\int_{\R}\left|F_{1}\right|dt\right)dx
&\leq C\left(s,R,\left\|\varphi\right\|_{C^{2}\left(\R^{N}\right)}\right)\left(\int_{\left|x\right|<2R}\left|u\left(x\right)\right|dx+\int_{\left|x\right|\geq 2R}\frac{\left|u\left(x\right)\right|}{\left|x\right|^{1+2s}}dx\right)\\
&\leq C\left(s,\beta_{1},R,\left\|\varphi\right\|_{C^{2}\left(\R^{N}\right)},\left\|u\right\|_{L^{\infty}\left(B_{2R}\right)}\right)\left(1+\int_{\left|x\right|\geq 2R}\frac{1}{\left|x\right|^{N+\delta_{2}}}dx\right)<+\infty.
\end{align*}

Next, we prove \eqref{E2-6}. Since $u\in C^{2}\left(\R^{N}\right)\cap\mathcal{L}_{s}$, there exists $\beta_{2}=\beta_{2}\left(N,s,\delta_{1},u\right)>0$ such that
\begin{align}\label{E2-10}
\left|u\left(x\right)\right|\leq\beta_{2}\left(1+\left|x\right|^{2s-\delta_{1}}\right),\quad\forall x\in\R^{N}.
\end{align}
For any $\left|x\right|<R$, by the conditions $u\in C^{2}\left(\R^{N}\right)\cap\mathcal{L}_{s}$ and \eqref{E2-10}, we infer that
\begin{align*}
&\frac{\left|u\left(x+te_{i}\right)+u\left(x-te_{i}\right)-2u\left(x\right)\right|}{\left|t\right|^{1+2s}}
\leq C\left(s,\delta_{1},\beta_{2},R,\left\|u\right\|_{C^{2}\left(B_{2R}\right)}\right)\times
\left\lbrace 
\begin{aligned}
&\frac{1}{\left|t\right|^{2s-1}},&&\left|t\right|<R,\\
&\frac{1}{\left|t\right|^{1+\delta_{1}}},&&\left|t\right|\geq R.
\end{aligned}
\right.
\end{align*}
Consequently, we get 
\begin{align}\label{E2-11}
&\int_{\R}\frac{\left|u\left(x+te_{i}\right)+u\left(x-te_{i}\right)-2u\left(x\right)\right|}{\left|t\right|^{1+2s}}dt\leq C\left(s,\delta_{1},\beta_{2},R,\left\|u\right\|_{C^{2}\left(B_{2R}\right)}\right),
\end{align}
and thus
\begin{align*}
\int_{\R^{N}}\left(\int_{\R}\left|F_{2}\right|dt\right)dx
=\int_{\left|x\right|<R}\left|\varphi\left(x\right)\right|\left(\int_{\R}\frac{\left|u\left(x+te_{i}\right)+u\left(x-te_{i}\right)-2u\left(x\right)\right|}{\left|t\right|^{1+2s}}dt\right)dx<+\infty.
\end{align*}
\end{proof}

Once we have established Lemma \ref{L2-8}, we can proceed with integration by parts for functions that belong to the space $C^{2}\left(\R^{N}\right)\cap\overline{\mathcal{L}_{s}}$.

\begin{cor}\label{C2-9}
Let $i\in\mathbf{S}_{N}$ and $\varphi\in C^{2}_{0}\left(\R^{N}\right)$. Then for any $u\in C^{2}\left(\R^{N}\right)\cap\overline{\mathcal{L}_{s}}$, we have
\begin{align*}
\int_{\R^{N}}u\mathcal{I}_{i}\varphi dx=\int_{\R^{N}}\varphi\mathcal{I}_{i}udx.
\end{align*}
\end{cor}
\begin{proof}

By Lemma \ref{L2-8} and $\overline{\mathcal{L}_{s}}\subset\mathcal{L}_{s}$, we apply Fubini's theorem to get 
\begin{align*}
\int_{\R^{N}}u\mathcal{I}_{i}\varphi dx
&=\int_{\R^{N}}\left(\int_{\R}\frac{\varphi\left(x+te_{i}\right)+\varphi\left(x-te_{i}\right)-2\varphi\left(x\right)}{\left|t\right|^{1+2s}}u\left(x\right)dt\right)dx\\
&=\int_{\R}\left(\int_{\R^{N}}\frac{\varphi\left(x+te_{i}\right)+\varphi\left(x-te_{i}\right)-2\varphi\left(x\right)}{\left|t\right|^{1+2s}}u\left(x\right)dx\right)dt\\
&=\int_{\R}\left(\int_{\R^{N}}\frac{u\left(x+te_{i}\right)+u\left(x-te_{i}\right)-2u\left(x\right)}{\left|t\right|^{1+2s}}\varphi\left(x\right)dx\right)dt\\
&=\int_{\R^{N}}\left(\int_{\R}\frac{u\left(x+te_{i}\right)+u\left(x-te_{i}\right)-2u\left(x\right)}{\left|t\right|^{1+2s}}\varphi\left(x\right)dt\right)dx\\
&=\int_{\R^{N}}\varphi\mathcal{I}_{i}udx.
\end{align*}
\end{proof}

Although Corollary \ref{C2-9} may not be applicable to functions in $C^{2}\left(\R^{N}\right)\cap\mathcal{L}_{s}$, the following integration by parts formula is valid. This can be interpreted as an integration by parts formula for $\left(-\Delta\right)^{s}$ in one dimension.

\begin{lem}\label{L2-10}
Let $i\in\mathbf{S}_{N}$ and $\varphi\in C^{2}_{0}\left(\R^{N}\right)$. Then for any $u\in C^{2}\left(\R^{N}\right)\cap\mathcal{L}_{s}$, we have
\begin{align*}
\int_{\R}u\mathcal{I}_{i}\varphi dx_{i}=\int_{\R}\varphi\mathcal{I}_{i}udx_{i}.
\end{align*}
\end{lem}

\begin{proof}
We assume that supp $\varphi\subseteq B_{R}$ with $R>1$. Let $x=\left(x_{i},x'_{i}\right)$. As the procedures in the proof of Lemma \ref{L2-8}, we define the maps $F_{3},F_{4}:\R\times\left(\R\backslash\left\{0\right\}\right)\rightarrow\R$ by
\begin{align*}
F_{3}\left(x_{i},t\right):=\frac{\varphi\left(x+te_{i}\right)+\varphi\left(x-te_{i}\right)-2\varphi\left(x\right)}{\left|t\right|^{1+2s}}u\left(x\right)
\end{align*}
and
\begin{align*}
F_{4}\left(x_{i},t\right):=\frac{u\left(x+te_{i}\right)+u\left(x-te_{i}\right)-2u\left(x\right)}{\left|t\right|^{1+2s}}\varphi\left(x\right)
\end{align*}
We need to show
\begin{align}\label{E2-12}
\int_{\R}\left(\int_{\R}\left|F_{3}\right|dt\right)dx_{i}<+\infty\quad\text{and}\quad\int_{\R}\left(\int_{\R}\left|F_{4}\right|dt\right)dx_{i}<+\infty.
\end{align}

By \eqref{E2-8} and \eqref{E2-10}, for any $u\in C^{2}\left(\R^{N}\right)\cap\mathcal{L}_{s}$, we obtain
\begin{align*}
&\int_{\R}\left(\int_{\R}\left|F_{3}\right|dt\right)dx_{i}\\
&\leq C\left(s,R,\left\|\varphi\right\|_{C^{2}\left(\R^{N}\right)}\right)\left(\int_{\left|x_{i}\right|<\sqrt{4R^{2}-\left|x'_{i}\right|^{2}}}\left|u\left(x\right)\right|dx_{i}+\int_{\left|x_{i}\right|\geq\sqrt{4R^{2}-\left|x'_{i}\right|^{2}}}\frac{\left|u\left(x\right)\right|}{\left|x\right|^{1+2s}}dx_{i}\right)\\
&\leq C\left(s,\beta_{2},R,\left|x'_{i}\right|,\left\|\varphi\right\|_{C^{2}\left(\R^{N}\right)},\left\|u\right\|_{L^{\infty}\left(B_{2R}\right)}\right)\left(1+\int_{\left|x_{i}\right|\geq\sqrt{4R^{2}-\left|x'_{i}\right|^{2}}}\frac{1+\left|x\right|^{2s-\delta_{1}}}{\left|x\right|^{1+2s}}dx_{i}\right)<+\infty.
\end{align*}
On the other hand, for any $\left|x\right|<R$, we have \eqref{E2-11}, therefore
\begin{align*}
&\int_{\R}\left(\int_{\R}\left|F_{4}\right|dt\right)dx_{i}\\
&=\int_{\left|x_{i}\right|<\sqrt{R^{2}-\left|x'_{i}\right|^{2}}}\left|\varphi\left(x\right)\right|\left(\int_{\R}\frac{\left|u\left(x+te_{i}\right)+u\left(x-te_{i}\right)-2u\left(x\right)\right|}{\left|t\right|^{1+2s}}dt\right)dx_{i}<+\infty.
\end{align*}
In virtue of \eqref{E2-12}, we get Lemma \ref{L2-10} by applying Fubini--Tonelli's theorem as the way stated in Corollary \ref{C2-9}.
\end{proof}

The integration by parts formula for functions in $C^{2}\left(\R_{+}^{N}\right)$ reveals increased complexity. While the function is smooth in $\R_{+}^{N}$, it may not maintain smoothness at the boundary $\partial\R_{+}^{N}$. We concentrate on functions $u\in C^{2}\left(\R_{+}^{N}\right)$ that do not has singularities at the boundary $\partial\R_{+}^{N}$, thereby satisfying
\begin{align}\label{E2-13}
\left\|u\right\|_{C^{2}\left(Q\cap\R^{N}_{+}\right)}<+\infty,\quad \forall \text{ compact }Q\subseteq\overline{\R_{+}^{N}}
\end{align}
and
\begin{align}\label{E2-14}
u=0\quad\text{in }\R^{N}\backslash\R_{+}^{N}.
\end{align}
For any $\varphi\in C^{2}_{0}\left(\R^{N}\right)$, we define
\begin{align*}
v_{\alpha}\left(x\right)=\left(x_{N}\right)_{+}^{\alpha}\varphi\left(x\right).
\end{align*}
The continuous function $v_{\alpha}$ is not smooth on $\partial\R_{+}^{N}$, nevertheless, we have

\begin{lem}\label{L2-11}
Let $i\in\mathbf{S}_{N}$, $\left(2s-1\right)_+<\alpha<2s$ and $\varphi\in C^{2}_{0}\left(\R^{N}\right)$. For any $u\in C^{2}\left(\R_{+}^{N}\right)\cap\overline{\mathcal{L}_{s}}$ that satisfies \eqref{E2-13} and \eqref{E2-14}, we have
\begin{align*}
\int_{\R^{N}}u\mathcal{I}_{i}v_{\alpha}dx=\int_{\R^{N}}v_{\alpha}\mathcal{I}_{i}udx.
\end{align*}
\end{lem}

\begin{proof}
We assume supp $\varphi\subseteq B_{R}$, $R>1$. Let us define two maps $F_{5},F_{6}:\R^{N}\times\left(\R\backslash\left\{0\right\}\right)\rightarrow\R$ by
\begin{align*}
F_{5}\left(x,t\right):=\frac{v_{\alpha}\left(x+te_{i}\right)+v_{\alpha}\left(x-te_{i}\right)-2v_{\alpha}\left(x\right)}{\left|t\right|^{1+2s}}u\left(x\right)
\end{align*}
and
\begin{align*}
F_{6}\left(x,t\right):=\frac{u\left(x+te_{i}\right)+u\left(x-te_{i}\right)-2u\left(x\right)}{\left|t\right|^{1+2s}}v_{\alpha}\left(x\right).
\end{align*}
We will show that
\begin{align}\label{E2-15}
\int_{\R^{N}}\left(\int_{\R}\left|F_{5}\right|dt\right)dx<+\infty
\end{align}
and
\begin{align}\label{E2-16}
\int_{\R^{N}}\left(\int_{\R}\left|F_{6}\right|dt\right)dx<+\infty.
\end{align}
Once \eqref{E2-15} and \eqref{E2-16} have been proved, we proceed with the proof of Corollary \ref{C2-9} in conjunction with Fubini--Tonelli's theorem to derive Lemma \ref{L2-11}.

Now we prove \eqref{E2-15}. Note that by \eqref{E2-14} we have
\begin{align}\label{E2-17}
\int_{\R^{N}\backslash\R_{+}^{N}}\left(\int_{\R}\left|F_{5}\right|dt\right)dx=0.
\end{align}
When $x\in\R^{N}_{+}\cap\left\{\left|x\right|\geq2R\right\}$, there is
\begin{align*}
\int_{\R}\frac{\left|v_{\alpha}\left(x+te_{i}\right)+v_{\alpha}\left(x-te_{i}\right)-2v_{\alpha}\left(x\right)\right|}{\left|t\right|^{1+2s}}dt
=2\int_{\left|x+te_{i}\right|\leq R}\frac{\left|v_{\alpha}\left(x+te_{i}\right)\right|}{\left|t\right|^{1+2s}}dt.
\end{align*}
Let us recall that 
\begin{align*}
\left|t\right|\geq\frac{\left|x\right|}{2}\quad\text{if }\left|x\pm te_{i}\right|\leq R,
\end{align*} 
thus we know
\begin{align*}
\int_{\left|x+te_{i}\right|\leq R}\frac{\left|v_{\alpha}\left(x+te_{i}\right)\right|}{\left|t\right|^{1+2s}}dt
&=\left\lbrace 
\begin{aligned}
&\int_{\left|x+te_{i}\right|\leq R}\frac{x_{N}^{\alpha}\left|\varphi\left(x+te_{i}\right)\right|}{\left|t\right|^{1+2s}}dt,&&i\in\mathbf{S}_{N-1}\\
&\int_{\left|x+te_{N}\right|\leq R}\frac{\left|\left(x_{N}+t\right)_{+}^{\alpha}\varphi\left(x+te_{N}\right)\right|}{\left|t\right|^{1+2s}}dt,&&i=N
\end{aligned}
\right.\\
&\leq \left(2^{1+2s}R^{\alpha}\left\|\varphi\right\|_{L^{\infty}\left(\R^{N}\right)}\int_{\left|x_{i}+t\right|\leq R}dt\right)\frac{1}{\left|x\right|^{1+2s}}.
\end{align*}
Therefore, when $x\in\R^{N}_{+}\cap\left\{\left|x\right|\geq2R\right\}$, there holds 
\begin{align}\label{E2-18}
\int_{\R}\frac{\left|v_{\alpha}\left(x+te_{i}\right)+v_{\alpha}\left(x-te_{i}\right)-2v_{\alpha}\left(x\right)\right|}{\left|t\right|^{1+2s}}dt
\leq C\left(s,\alpha,R,\left\|\varphi\right\|_{L^{\infty}\left(\R^{N}\right)}\right)\frac{1}{\left|x\right|^{1+2s}}.
\end{align}
When $x\in\R^{N}_{+}\cap\left\{\left|x\right|<2R\right\}$, if $i\in\mathbf{S}_{N-1}$, 
\begin{align*}
&\int_{\R}\frac{\left|v_{\alpha}\left(x+te_{i}\right)+v_{\alpha}\left(x-te_{i}\right)-2v_{\alpha}\left(x\right)\right|}{\left|t\right|^{1+2s}}dt\\
&=\int_{\R}\frac{\left|x_{N}^{\alpha}\varphi\left(x+te_{i}\right)+x_{N}^{\alpha}\varphi\left(x-te_{i}\right)-2x_{N}^{\alpha}\varphi\left(x\right)\right|}{\left|t\right|^{1+2s}}dt\\
&=x_{N}^{\alpha}\int_{\R}\frac{\left|\varphi\left(x+te_{i}\right)+\varphi\left(x-te_{i}\right)-2\varphi\left(x\right)\right|}{\left|t\right|^{1+2s}}dt,
\end{align*}
we get by \eqref{E2-7} and $\alpha<2s$ that
\begin{align}\label{E2-19}
\begin{split}
\int_{\R}\frac{\left|v_{\alpha}\left(x+te_{i}\right)+v_{\alpha}\left(x-te_{i}\right)-2v_{\alpha}\left(x\right)\right|}{\left|t\right|^{1+2s}}dt
&\leq C\left(s,R,\left\|\varphi\right\|_{C^{2}\left(\R^{N}\right)}\right)x_{N}^{\alpha}\\
&\leq C\left(s,\alpha,R,\left\|\varphi\right\|_{C^{2}\left(\R^{N}\right)}\right)x_{N}^{\alpha-2s}.
\end{split}
\end{align}
If $i=N$, let $t=x_{N}r$, we obtain that
\begin{align*}
&\int_{\R}\frac{\left|v_{\alpha}\left(x+te_{N}\right)+v_{\alpha}\left(x-te_{N}\right)-2v_{\alpha}\left(x\right)\right|}{\left|t\right|^{1+2s}}dt\\
&=\int_{\R}\frac{\left|\left(x_{N}+t\right)_{+}^{\alpha}\varphi\left(x+te_{N}\right)+\left(x_{N}-t\right)_{+}^{\alpha}\varphi\left(x-te_{N}\right)-2x_{N}^{\alpha}\varphi\left(x\right)\right|}{\left|t\right|^{1+2s}}dt\\
&=x_{N}^{\alpha-2s}\int_{\R}\frac{\left|\left(1+r\right)_{+}^{\alpha}\varphi\left(x+x_{N}re_{N}\right)+\left(1-r\right)_{+}^{\alpha}\varphi\left(x-x_{N}re_{N}\right)-2\varphi\left(x\right)\right|}{\left|r\right|^{1+2s}}dr.
\end{align*}
We define the function $B:\R^{N}\times\R\rightarrow\R$ by
\begin{align*}
B\left(x,r\right):=\left(1+r\right)_{+}^{\alpha}\varphi\left(x+x_{N}re_{N}\right)+\left(1-r\right)_{+}^{\alpha}\varphi\left(x-x_{N}re_{N}\right)-2\varphi\left(x\right).
\end{align*}
For any $\alpha>0$, one can verify that 
\begin{align*}
B\in C^{2}\left(\R^{N}\times\left\{\left|r\right|\leq\frac{1}{2}\right\}\right)
\end{align*}
satisfies
\begin{align*}
\left\lbrace 
\begin{aligned}
&B\left(x,0\right)=0,&&\forall x\in\R^{N},\\
&B\left(x,r\right)=B\left(x,-r\right),&&\forall\left(x,r\right)\in\R^{N}\times\R.
\end{aligned}
\right.
\end{align*}
Hence for any $x\in\R^{N}_{+}\cap\left\{\left|x\right|<2R\right\}$, we obtain that
\begin{align*}
\left|B\left(x,r\right)\right|
\leq C\left(\alpha,R,\left\|\varphi\right\|_{C^{2}\left(\R^{N}\right)}\right)\times
\left\lbrace 
\begin{aligned}
&\left|r\right|^{2},&&\left|r\right|<\frac{1}{2},\\
&\left|r\right|^{\alpha},&&\left|r\right|\geq\frac{1}{2}.
\end{aligned}
\right.
\end{align*}
Therefore, the condition $\alpha<2s$ implies
\begin{align}\label{E2-20}
\begin{split}
\int_{\R}\frac{\left|v_{\alpha}\left(x+te_{N}\right)+v_{\alpha}\left(x-te_{N}\right)-2v_{\alpha}\left(x\right)\right|}{\left|t\right|^{1+2s}}dt
&=x_{N}^{\alpha-2s}\int_{\R}\frac{\left|B\left(x,r\right)\right|}{\left|r\right|^{1+2s}}dr\\
&\leq C\left(s,\alpha,R,\left\|\varphi\right\|_{C^{2}\left(\R^{N}\right)}\right)x_{N}^{\alpha-2s}.
\end{split}
\end{align}
Combining \eqref{E2-18}, \eqref{E2-19} and \eqref{E2-20}, we conclude that
\begin{align}\label{E2-21}
\begin{split}
&\int_{\R}\frac{\left|v_{\alpha}\left(x+te_{i}\right)+v_{\alpha}\left(x-te_{i}\right)-2v_{\alpha}\left(x\right)\right|}{\left|t\right|^{1+2s}}dt\\
&\leq C\left(s,\alpha,R,\left\|\varphi\right\|_{C^{2}\left(\R^{N}\right)}\right)\times
\left\lbrace 
\begin{aligned}
&\frac{1}{x_{N}^{2s-\alpha}},&&x\in\R^{N}_{+}\cap\left\{\left|x\right|<2R\right\},\\
&\frac{1}{\left|x\right|^{1+2s}},&&x\in\R^{N}_{+}\cap\left\{\left|x\right|\geq2R\right\}.
\end{aligned}
\right.
\end{split}
\end{align}
Since that $u\in C^{2}\left(\R_{+}^{N}\right)\cap\overline{\mathcal{L}_{s}}$ satisfies \eqref{E2-13} and \eqref{E2-14}, we can deduce that there exists $\beta_{3}=\beta_{3}\left(N,s,\delta_{2},R,u\right)>0$ such that
\begin{align}\label{E2-22}
\left|u\left(x\right)\right|\leq\beta_{3}\left|x\right|^{2s-N+1-\delta_{2}},\quad\forall\left|x\right|\geq2R.
\end{align}
By \eqref{E2-17}, \eqref{E2-21}, \eqref{E2-22} and $\alpha>2s-1$, for any $u\in C^{2}\left(\R_{+}^{N}\right)\cap\overline{\mathcal{L}_{s}}$ satisfies \eqref{E2-13} and \eqref{E2-14},
\begin{align*}
&\int_{\R^{N}}\left(\int_{\R}\left|F_{5}\right|dt\right)dx\\
&\leq C\left(s,\alpha,R,\left\|\varphi\right\|_{C^{2}\left(\R^{N}\right)}\right)\left(\int_{\R^{N}_{+}\cap\left\{\left|x\right|<2R\right\}}\frac{\left|u\left(x\right)\right|}{x_{N}^{2s-\alpha}}dx+\int_{\R^{N}_{+}\cap\left\{\left|x\right|\geq2R\right\}}\frac{\left|u\left(x\right)\right|}{\left|x\right|^{1+2s}}dx\right)\\
&\leq C\left(s,\alpha,\beta_{3},R,\left\|\varphi\right\|_{C^{2}\left(\R^{N}\right)},\left\|u\right\|_{L^{\infty}\left(B_{2R}\right)}\right)\\
&\quad\times\left(\int_{\R^{N}_{+}\cap\left\{\left|x\right|<2R\right\}}\frac{1}{x_{N}^{2s-\alpha}}dx+\int_{\R^{N}_{+}\cap\left\{\left|x\right|\geq2R\right\}}\frac{1}{\left|x\right|^{N+\delta_{2}}}dx\right)<+\infty.
\end{align*}

We next prove \eqref{E2-16}. We see that
\begin{align}\label{E2-23}
\int_{\left(\R^{N}\backslash\R_{+}^{N}\right)\cup\left\{\left|x\right|\geq R\right\}}\left(\int_{\R}\left|F_{6}\right|dt\right)dx=0.
\end{align}
Since that $\overline{\mathcal{L}_{s}}\subset\mathcal{L}_{s}$ and thus $u\in C^{2}\left(\R_{+}^{N}\right)\cap\mathcal{L}_{s}$, there exists $\beta_{4}=\beta_{4}\left(N,s,\delta_{1},u\right)>0$ such that
\begin{align}\label{E2-24}
\left|u\left(x\right)\right|\leq\beta_{4}\left(1+\left|x\right|^{2s-\delta_{1}}\right),\quad\forall x\in\R^{N}.
\end{align}
For any $x\in\R^{N}_{+}\cap\left\{\left|x\right|<R\right\}$, we obtain by \eqref{E2-13} that $u\in C^{2}\left(\overline{\R^{N}_{+}}\cap B_{2R}\right)$, hence when $\left|t\right|<x_{N}$, there is $x\pm te_{i}\in B_{x_{N}}\left(x\right)\subset B_{2R}$, together with \eqref{E2-24} we have
\begin{align*}
&\frac{\left|u\left(x+te_{i}\right)+u\left(x-te_{i}\right)-2u\left(x\right)\right|}{\left|t\right|^{1+2s}}\\
&\leq C\left(s,\delta_{1},\beta_{4},R,\left\|u\right\|_{C^{2}\left(\overline{\R^{N}_{+}}\cap B_{2R}\right)}\right)\times
\left\lbrace 
\begin{aligned}
&\frac{1}{\left|t\right|^{2s-1}},&&\left|t\right|<x_{N},\\
&\frac{1+\left|t\right|^{2s-\delta_{1}}}{\left|t\right|^{1+2s}},&&\left|t\right|\geq x_{N}.
\end{aligned}
\right.
\end{align*}
Consequently, for any $x\in\R^{N}_{+}\cap\left\{\left|x\right|<R\right\}$, we infer that
\begin{align}\label{E2-25}
\begin{split}
&\int_{\R}\frac{\left|u\left(x+te_{i}\right)+u\left(x-te_{i}\right)-2u\left(x\right)\right|}{\left|t\right|^{1+2s}}dt\\
&\leq C\left(s,\delta_{1},\beta_{4},R,\left\|u\right\|_{C^{2}\left(\overline{\R^{N}_{+}}\cap B_{2R}\right)}\right)\left(\int_{\left|t\right|<x_{N}}\frac{1}{\left|t\right|^{2s-1}}dt+\int_{\left|t\right|\geq x_{N}}\frac{1+\left|t\right|^{2s-\delta_{1}}}{\left|t\right|^{1+2s}}dt\right)\\
&=C\left(s,\delta_{1},\beta_{4},R,\left\|u\right\|_{C^{2}\left(\overline{\R^{N}_{+}}\cap B_{2R}\right)}\right)\left(x_{N}^{2-2s}+\frac{1}{x_{N}^{2s}}+\frac{1}{x_{N}^{\delta_{1}}}\right)\\
&\leq C\left(s,\delta_{1},\beta_{4},R,\left\|u\right\|_{C^{2}\left(\overline{\R^{N}_{+}}\cap B_{2R}\right)}\right)\frac{1}{x_{N}^{2s}}.
\end{split}
\end{align}
Finally we deduce by \eqref{E2-23}, \eqref{E2-25} and $\alpha>2s-1$ that
\begin{align*}
&\int_{\R^{N}}\left(\int_{\R}\left|F_{6}\right|dt\right)dx\\
&\leq C\left(s,\delta_{1},\beta_{4},R,\left\|\varphi\right\|_{L^{\infty}\left(\R^{N}\right)},\left\|u\right\|_{C^{2}\left(\overline{\R^{N}_{+}}\cap B_{2R}\right)}\right)\int_{\R^{N}_{+}\cap\left\{\left|x\right|<R\right\}}\frac{1}{x_{N}^{2s-\alpha}}dx<+\infty.
\end{align*}
\end{proof}

\subsection{Fourier symbol}
We denote the $\widehat{u}$ to represent the Fourier transform
\begin{align*}
\widehat{u}\left(\xi\right)=\int_{\R^{N}}u\left(x\right)e^{-2\pi ix\cdot\xi}dx,
\end{align*}
while $u^{\vee}$ signifies its inverse. Let us recall that $L$ is the infinitesimal generator of the heat kernel 
\begin{align*}
e^{-tL}=\left(e^{-\widehat{L}t}\right)^{\vee}
\end{align*}
under the action on $u$ by
\begin{align*}
e^{-tL}u\left(x\right)=\int_{\R^{N}}e^{-tL}\left(x,y\right)u\left(y\right)dy,
\end{align*}
where $\widehat{L}$ is the $2s$-order Fourier symbol of $L$ defined by 
\begin{align*}
\widehat{L}\left(\xi\right)=\int_{\mathbb{S}^{N-1}}\left|2\pi\xi\cdot\theta\right|^{2s}d\mu.
\end{align*}
For the convenience of readers, for any $u\in\mathscr{S}$, one can take the Fourier transform on $Lu$ to obtain
\begin{align*}
\widehat{-Lu}\left(\xi\right)
&=-\frac{C_{s}}{2}\int_{\mathbb{S}^{N-1}}\int_{\R}\frac{\widehat{u\left(x+t\theta\right)}+\widehat{u\left(x-t\theta\right)}-2\widehat{u\left(x\right)}}{\left|t\right|^{1+2s}}dtd\mu\\
&=C_{s}\widehat{u}\left(\xi\right)\int_{\mathbb{S}^{N-1}}\int_{\R}\frac{1-\cos\left(2\pi t\xi\cdot\theta\right)}{\left|t\right|^{1+2s}}dtd\mu\\
&=C_{s}\left(\int_{\R}\frac{1-\cos x}{\left|x\right|^{1+2s}}dx\right)\left(\int_{\mathbb{S}^{N-1}}\left|2\pi\xi\cdot\theta\right|^{2s}d\mu\right)\widehat{u}\left(\xi\right)\\
&=\widehat{L}\left(\xi\right)\widehat{u}\left(\xi\right).
\end{align*}
Therefore, under the condition 
\begin{align*}
\mu=\sum^{N}_{i=1}\left(\delta_{e_{i}}+\delta_{-e_{i}}\right),
\end{align*}
we derive that the Fourier symbol of $\mathcal{I}$ is given by
\begin{align*}
\widehat{\mathcal{I}}\left(\xi\right)=\sum^{N}_{i=1}\left|2\pi\xi_{i}\right|^{2s}.
\end{align*}
This Fourier symbol $\widehat{\mathcal{I}}$ is radially symmetric if and only if $s=1$, in which case $\left|2\pi\xi\right|^{2}$ corresponds to the Fourier symbol of $\Delta$. In the context of fractional calculus $0<s<1$, for any $i,j\in\mathbf{S}_{N}$ with $i<j$, we obtain that $\widehat{\mathcal{I}}$ is symmetric with respect to $P_{e_{i},0}$, i.e.
\begin{align*}
\widehat{\mathcal{I}}\left(\xi\right)=\widehat{\mathcal{I}}\left(\left|\xi_{1}\right|,...,\left|\xi_{N}\right|\right).
\end{align*}
Furthermore, the symbol $\widehat{\mathcal{I}}$ is also symmetric with respect to $P_{e_{i}\pm e_{j},0}$, allowing for the interchange of the positions of $\xi_{i}$ and $\xi_{j}$ such that
\begin{align*}
\widehat{\mathcal{I}}\left(\left|\xi_{1}\right|,...,\left|\xi_{i}\right|,...,\left|\xi_{j}\right|,...,\left|\xi_{N}\right|\right)=\widehat{\mathcal{I}}\left(\left|\xi_{1}\right|,...,\left|\xi_{j}\right|,...,\left|\xi_{i}\right|,...,\left|\xi_{N}\right|\right).
\end{align*}

The operator $\mathcal{I}$ can be characterized as a pseudo-differential operator with the symbol $\widehat{\mathcal{I}}$, as expressed in the following manner:
\begin{align}\label{E2-26}
-\mathcal{I}u=\left(\left|2\pi\xi\right|_{2s}^{2s}\widehat{u}\left(\xi\right)\right)^{\vee}
\end{align}
with
\begin{align*}
\left|\xi\right|_{2s}=\left(\sum^{N}_{i=1}\left|\xi_{i}\right|^{2s}\right)^{\frac{1}{2s}}.
\end{align*}

\subsection{Symmetric potential}\label{S2-5}
With regard of the Fourier symbol $\widehat{\mathcal{I}}$, we define the potential
\begin{align}\label{E2-27}
G_{s}\left(x\right):=\left(\left|2\pi\xi\right|^{-2s}_{2s}\right)^{\vee}\left(x\right),\quad x\neq0,
\end{align}
which is the fundamental solution of $\mathcal{I}u=0$. We reformulate the definition \eqref{E2-27} to be
\begin{align}\label{E2-28}
G_{s}\left(x\right)
=\int_{\R^{N}}\left|2\pi\xi\right|^{-2s}_{2s}e^{2\pi ix\cdot\xi}d\xi
=\frac{1}{\left(2\pi\right)^{N}}\int_{\R^{N}}\left|\xi\right|^{-2s}_{2s}\cos\left(x\cdot\xi\right)d\xi.
\end{align}
It is evident that $G_{s}$ does not represent a radially symmetric potential when $0<s<1$ due to the absence of rotational invariance. Let us define an $N\times N$ reflection matrix $\mathbf{M}$ with respect to the plane $P_{e_{i},0}$. By substituting $\zeta=\mathbf{M}^{T}\xi$, we derive that
\begin{align*}
G_{s}\left(\mathbf{M}x\right)
&=\frac{1}{\left(2\pi\right)^{N}}\int_{\R^{N}}\left|\xi\right|^{-2s}_{2s}\cos\left(\mathbf{M}x\cdot\xi\right)d\xi\\
&=\frac{1}{\left(2\pi\right)^{N}}\int_{\R^{N}}\left|\xi\right|^{-2s}_{2s}\cos\left(x\cdot \mathbf{M}^{T}\xi\right)d\xi\\
&=\frac{1}{\left(2\pi\right)^{N}}\int_{\R^{N}}\left|\zeta\right|^{-2s}_{2s}\cos\left(x\cdot\zeta\right)d\zeta\\
&=G_{s}\left(x\right).
\end{align*}
Therefore $G_{s}$ is symmetric with respect to $P_{e_{i},0}$. Employing a similar methodology, it can be demonstrated that $G_{s}$ is also symmetric with respect to $P_{e_{i}\pm e_{j},0}$. Consequently, for any $i,j\in\mathbf{S}_{N}$ with $i<j$,
\begin{align*}
G_{s}\left(x\right)=G\left(\left|x_{1}\right|,...,\left|x_{i}\right|,...,\left|x_{j}\right|,...,\left|x_{N}\right|\right)=G\left(\left|x_{1}\right|,...,\left|x_{j}\right|,...,\left|x_{i}\right|,...,\left|x_{N}\right|\right).
\end{align*} 

For any $\lambda>0$, the potential $G_{s}$ is homogeneous, satisfying
\begin{align}\label{E2-29}
G_{s}\left(\lambda x\right)=\lambda^{2s-N}G_{s}\left(x\right).
\end{align}
The identity \eqref{E2-29} yields $\left|G_{s}\right|\asymp\left|x\right|^{2s-N}$ (the symbol $\asymp$ denotes equivalence). Moreover, by the formula  
\begin{align*}
\max_{i\in\mathbf{S}_{N}}\left|x_{i}\right|\leq\left|x\right|_{2s}\leq N^{\frac{1}{2s}}\max_{i\in\mathbf{S}_{N}}\left|x_{i}\right|,
\end{align*}
we get 
\begin{align*}
\left|G_{s}\right|\asymp\left|x\right|^{2s-N}\asymp\left|x\right|_{2s}^{2s-N}.
\end{align*}
Let us further apply Fubini's theorem along with the formula
\begin{align*}
\frac{1}{x}=\int^{+\infty}_{0}e^{-xy}dy,\quad x>0
\end{align*}
to \eqref{E2-28}, we obtain a more explicit formulation of $G_{s}$:
\begin{align}\label{E2-30}
\begin{split}
G_{s}\left(x\right)
&=\int_{\R^{N}}\left(\int^{+\infty}_{0}e^{-\left|2\pi\xi\right|^{2s}_{2s}y+2\pi ix\cdot\xi}dy\right)d\xi\\
&=\pi^{-N}\int^{+\infty}_{0}\left(\prod_{i=1}^{N}\int_{0}^{+\infty}e^{-yt^{2s}}\cos\left(x_{i}t\right)dt\right)dy.
\end{split}
\end{align}
For any $\left(a,b\right)\in\R_{+}\times\R$, the integral value is given by
\begin{align}\label{E2-31}
\int_{0}^{+\infty}e^{-at^{2s}}\cos\left(bt\right)dt=\left\lbrace 
\begin{aligned}
&\frac{a}{a^{2}+b^{2}},&&s=\frac{1}{2},\\
&\sqrt{\frac{\pi}{4a}}e^{-\frac{b^{2}}{4a}},&&s=1.
\end{aligned}
\right.
\end{align}
Therefore, let us insert \eqref{E2-31} into \eqref{E2-30}, we have
\begin{align}\label{E2-32}
G_{s}\left(x\right)=\left\lbrace 
\begin{aligned}
&\pi^{-N}\int^{+\infty}_{0}\left(\prod_{i=1}^{N}\frac{y}{y^{2}+x_{i}^{2}}\right)dy,&&s=\frac{1}{2},N\geq2.\\
&\frac{\Gamma\left(\frac{N}{2}-1\right)}{4\pi^{\frac{N}{2}}}\cdot\frac{1}{\left|x\right|^{N-2}},&&s=1,N>2.
\end{aligned}
\right.
\end{align}
Moreover, we have
\begin{align*}
G_{\frac{1}{2}}\left(x\right)=\widehat{G_{\frac{1}{2}}}\left(x\right)=\frac{1}{2\pi\left(\left|x_{1}\right|+\left|x_{2}\right|\right)},\quad N=2.
\end{align*}

When $N>2$ and $s=1$, since the function $e^{-|x|^{2}}$ is Gaussian function, thus $G_{1}$ corresponds to the well-known Newtonian potential $\left|\cdot\right|^{2-N}$ (In the case of the fractional Laplacian $\left(-\Delta\right)^{s}$, it is associated with the Riesz potential $\left|\cdot\right|^{2s-N}$). However, for $s\neq\frac{1}{2}$, the integral value \eqref{E2-31} remains undetermined, which consequently renders the potential $G_{s}$ also unknown. 

The positivity of $G_{s}$ is an intriguing question. It is obvious that $G_{\frac{1}{2}}>0$. We guess that $G_{s}$ is a positive potential for $0<s\leq1$, but we can only establish $G_{s}>0$ in the specific case $0<s\leq\frac{1}{2}$. We get the positivity of $G_{s}$ by combining \eqref{E2-30} with the following lemma.

\begin{lem}
Let $a>0$ and $0<s\leq1$. We have
\begin{align*}
\left\lbrace 
\begin{aligned}
\int^{+\infty}_{0}e^{-ax^{2s}}\sin xdx&>0,&&0<s\leq1,\\
\int^{+\infty}_{0}e^{-ax^{2s}}\cos xdx&>0,&&0<s\leq\frac{1}{2}.
\end{aligned}
\right.
\end{align*}
\end{lem}
\begin{proof}
We first claim that for any strictly decreasing function $f:\left[0,+\infty\right)\rightarrow\R$, there holds
\begin{align}\label{E2-33}
\int^{+\infty}_{0}f\left(x\right)\sin xdx>0.
\end{align}
Indeed, for any $K\in\N_{+}$, there is
\begin{align*}
\int^{+\infty}_{0}f\left(x\right)\sin xdx
\geq\liminf_{K\rightarrow+\infty}\sum_{k=0}^{K}\int^{2\pi+2k\pi}_{2k\pi}f\left(x\right)\sin xdx.
\end{align*}
For any $k\in\N$, we obtain that
\begin{align*}
\int^{2\pi+2k\pi}_{2k\pi}f\left(x\right)\sin xdx
&=\int^{\pi+2k\pi}_{2k\pi}f\left(x\right)\sin xdx+\int^{2\pi+2k\pi}_{\pi+2k\pi}f\left(x\right)\sin xdx\\
&>f\left(\pi+2k\pi\right)\left(\int^{\pi+2k\pi}_{2k\pi}\sin xdx+\int^{2\pi+2k\pi}_{\pi+2k\pi}\sin xdx\right)=0,
\end{align*}
hence we have \eqref{E2-33}. Let $f\left(x\right)=e^{-ax^{2s}}$ in \eqref{E2-33}, we prove the first integral is positive. 

When $0<s\leq\frac{1}{2}$, the integration by parts formula yields
\begin{align*}
\int^{+\infty}_{0}e^{-ax^{2s}}\cos xdx
&=2sa\int^{+\infty}_{0}e^{-ax^{2s}}x^{2s-1}\sin xdx.
\end{align*}
The positivity is maintained if we let $f\left(x\right)=e^{-ax^{2s}}x^{2s-1}$ in \eqref{E2-33}.
\end{proof}

\subsection{Equivalent integral equation}\label{S2-6}

The equivalence of \eqref{E1-3} and 
\begin{align}\label{E2-34}
u\left(x\right)=\int_{\R^{N}}\frac{u^{p}\left(y\right)}{\left|x-y\right|^{N-2s}}dy,\quad x\in\R^{N}
\end{align}
was studied by Chen--Li--Ou \cite{CLO}. The form \eqref{E2-34} is the Green's representation for the solution of \eqref{E1-3}. The positive solution $u\in L^{\infty}_{loc}\left(\R^{N}\right)\cap\mathcal{L}_{s}$ of \eqref{E1-3} satisfies \eqref{E2-34}, as well as the positive solution of \eqref{E2-34} satisfies \eqref{E1-3} in the distributional sense. Furthermore, the equations \eqref{E1-3} and \eqref{E2-34} are equivalent in the weak sense.

When $p>\frac{N}{N-2s}$, if $u$ is a positive classical solution of \eqref{E1-7} and $u$ satisfies \eqref{E1-9}, it follows that $u\in L^{\frac{N\left(p-1\right)}{2s}}\left(\R^{N}\right)$. By applying the Fourier transform to \eqref{E1-7}, we derive by \eqref{E2-26} that
\begin{align*}
\left|2\pi\xi\right|_{2s}^{2s}\widehat{u}\left(\xi\right)=\widehat{u^{p}}\left(\xi\right),
\end{align*}
then we have
\begin{align*}
u=\left(\widehat{u^{p}}\left(\xi\right)\left|2\pi\xi\right|_{2s}^{-2s}\right)^{\vee}=u^{p}\ast G_{s}.
\end{align*}
On the other hand, suppose $u\in L^{\frac{N\left(p-1\right)}{2s}}\left(\R^{N}\right)$ is a positive solution of \eqref{E1-10}, then for any $\varphi\in C^{\infty}_{0}\left(\R^{N}\right)$, we have
\begin{align*}
\int_{\R^{N}}u\left(-\mathcal{I}\varphi\right)dx
&=\int_{\R^{N}}\left(u^{p}\ast G_{s}\right)\left(-\mathcal{I}\varphi\right)dx\\
&=\int_{\R^{N}}\left(\left(-\mathcal{I}\varphi\right)\ast G_{s}\right)u^{p}dx\\
&=\int_{\R^{N}}u^{p}\varphi dx,
\end{align*}
we obtain that $u$ is a distributional solution of \eqref{E1-7}.

Since that $\left|\xi\right|_{2s}\asymp\left|\xi\right|$, we can define the weak solutions $u\in H^{s}\left(\R^{N}\right)$ of \eqref{E1-7} by 
\begin{align*}
\int_{\R^{N}}\left|\xi\right|_{2s}^{2s}\widehat{u}\left(\xi\right)\widehat{v}\left(\xi\right)d\xi=\int_{\R^{N}}u^{p}\left(x\right)v\left(x\right)dx,\quad\forall v\in C^{\infty}_{0}\left(\R^{N}\right).
\end{align*}
The phenomenon of equivalence observed between \eqref{E1-7} and \eqref{E1-10} can be demonstrated in the weak sense through a similar argument presented by Chen--Li--Ou \cite{CLO}.

\section{Nonexistence and symmetry in the whole space}\label{S3}
\subsection{Nonexistence of positive supersolutions: proof of Theorem \ref{T1-1} and \ref{T1-2}}
\subsubsection{Proof of Theorem \ref{T1-1}}
We choose $\varphi\in C^{\infty}_{0}\left(\R^{N}\right)$ satisfying $0<\varphi\leq1$ in $B_{2}$ and
\begin{align*}
\varphi\left(x\right)=
\left\lbrace 
\begin{aligned}
&1,&&x\in B_{1},\\
&0,&&x\notin B_{2}.
\end{aligned}
\right.
\end{align*}
We assert that there exists $M>0$ such that
\begin{align}\label{E3-1}
-\mathcal{I}\varphi\left(x\right)\leq M\varphi\left(x\right),\quad\forall x\in\R^{N}.
\end{align}
Since that for any $\left|x\right|\geq2$, we have
\begin{align*}
-\mathcal{I}\varphi\left(x\right)
=-\sum^{N}_{i=1}\int_{\R}\frac{\varphi\left(x+te_{i}\right)+\varphi\left(x-te_{i}\right)}{\left|t\right|^{1+2s}}dt\leq0=M\varphi\left(x\right),
\end{align*}
hence we only need to show 
\begin{align}\label{E3-2}
\inf_{x\in B_{2}}\frac{\mathcal{I}\varphi\left(x\right)}{\varphi\left(x\right)}>-\infty.
\end{align}
We assume \eqref{E3-2} does not hold, then there exists a convergent sequence $\left\{x_{n}\right\}\subset B_{2}$ such that
\begin{align}\label{E3-3}
\lim_{n\rightarrow+\infty}\frac{\mathcal{I}\varphi\left(x_{n}\right)}{\varphi\left(x_{n}\right)}=-\infty.
\end{align}
Let $x_{n}\rightarrow x_{\infty}\in\overline{B_{2}}$ as $n\rightarrow+\infty$. We consider the following two cases. 

Case 1: $\left|x_{\infty}\right|<2$. We obtain that
\begin{align*}
\lim_{n\rightarrow+\infty}\frac{\mathcal{I}\varphi\left(x_{n}\right)}{\varphi\left(x_{n}\right)}=\frac{\mathcal{I}\varphi\left(x_{\infty}\right)}{\varphi\left(x_{\infty}\right)},
\end{align*}
which contradicts \eqref{E3-3}.

Case 2: $\left|x_{\infty}\right|=2$. In this case, we have
\begin{align*}
\mathcal{I}\varphi\left(x_{\infty}\right)
=\sum^{N}_{i=1}\int_{\R}\frac{\varphi\left(x_{\infty}+te_{i}\right)+\varphi\left(x_{\infty}-te_{i}\right)}{\left|t\right|^{1+2s}}dt\geq0.
\end{align*}
We assume $\mathcal{I}\varphi\left(x_{\infty}\right)=0$, then it follows that
\begin{align*}
\varphi\left(x_{\infty}\pm te_{i}\right)=0,\quad\forall i\in\mathbf{S}_{N},\forall t\in\R.
\end{align*}
But this can not occur because $\left|x_{\infty}\right|=2$, so we infer that
\begin{align*}
\lim_{n\rightarrow+\infty}\mathcal{I}\varphi\left(x_{n}\right)
=\mathcal{I}\varphi\left(x_{\infty}\right)>0.
\end{align*}
Combining with the fact that 
\begin{align*}
\varphi\left(x_{n}\right)>0\quad\text{and}\quad\lim_{n\rightarrow+\infty}\varphi\left(x_{n}\right)=0,
\end{align*}
we get
\begin{align*}
\lim_{n\rightarrow+\infty}\frac{\mathcal{I}\varphi\left(x_{n}\right)}{\varphi\left(x_{n}\right)}=+\infty.
\end{align*}
This is another contradiction, consequently, the assertion \eqref{E3-2} is valid.

For any $R>0$, we make the rescaled test-function
\begin{align*}
\varphi_{R}\left(x\right)=\varphi\left(\frac{x}{R}\right).
\end{align*}
By Lemma \ref{L2-1}, we have
\begin{align}\label{E3-4}
\mathcal{I}\varphi_{R}\left(x\right)=R^{-2s}\mathcal{I}\varphi\left(\frac{x}{R}\right).
\end{align}
Since $u\in C^{2}\left(\R^{N}\right)\cap\overline{\mathcal{L}_{s}}$ is a nonnegative solution of \eqref{E1-5}, hence by Corollary \ref{C2-9}, $u$ also satisfies \eqref{E1-6}. Therefore \eqref{E3-1} and \eqref{E3-4} imply that
\begin{align}\label{E3-5}
\int_{\R^{N}}u^{p}\varphi_{R}dx
\leq\int_{\R^{N}}u\left(-\mathcal{I}\varphi_{R}\right)dx
\leq MR^{-2s}\int_{\R^{N}}u\varphi_{R}dx.
\end{align}
By the H\"{o}lder inequality, there holds 
\begin{align*}
\int_{\R^{N}}u\varphi_{R}dx
\nonumber&\leq\left(\int_{B_{2R}}u^{p}\varphi_{R}dx\right)^{\frac{1}{p}}\left(\int_{B_{2R}}\varphi_{R}dx\right)^{\frac{p-1}{p}}\\
&\leq C\left(N,p\right)R^{\frac{N\left(p-1\right)}{p}}\left(\int_{\R^{N}}u^{p}\varphi_{R}dx\right)^{\frac{1}{p}}.
\end{align*}
Finally we obtain by \eqref{E3-5} that for any $p>1$,
\begin{align}\label{E3-6}
\int_{\R^{N}}u^{p}\varphi_{R}dx\leq C\left(N,M,p\right)R^{N-\frac{2sp}{p-1}}.
\end{align}

When $1<p<\frac{N}{N-2s}$, we get $u\equiv0$ by letting $R\rightarrow+\infty$ in \eqref{E3-6}. When $p=\frac{N}{N-2s}$, as $R\rightarrow+\infty$, we know from \eqref{E3-6} that 
\begin{align}\label{E3-7}
\int_{\R^{N}}u^{p}dx<+\infty.
\end{align}
Let us rewrite
\begin{align*}
\int_{\R^{N}}u\varphi_{R}dx
=\int_{\left|x\right|\leq\sqrt{R}}u\varphi_{R}dx+\int_{\sqrt{R}\leq\left|x\right|\leq 2R}u\varphi_{R}dx.
\end{align*}
By the H\"{o}lder inequality, we have
\begin{align*}
\int_{\left|x\right|\leq\sqrt{R}}u\varphi_{R}dx
&\leq C\left(N,p\right)R^{s}\left(\int_{\R^{N}}u^{p}dx\right)^{\frac{1}{p}}
\end{align*}
and
\begin{align*}
\int_{\sqrt{R}\leq\left|x\right|\leq 2R}u\varphi_{R}dx
&\leq C\left(N,p\right)R^{2s}\left(\int_{\sqrt{R}\leq\left|x\right|\leq 2R}u^{p}dx\right)^{\frac{1}{p}},
\end{align*}
there follows by \eqref{E3-5} that
\begin{align*}
\int_{\R^{N}}u^{p}\varphi_{R}dx\leq C\left(N,M,p\right)\left[R^{-s}\left(\int_{\R^{N}}u^{p}dx\right)^{\frac{1}{p}}+\left(\int_{\sqrt{R}\leq\left|x\right|\leq 2R}u^{p}dx\right)^{\frac{1}{p}}\right].
\end{align*}
We again get $u\equiv0$ by letting $R\rightarrow+\infty$ as indicated by \eqref{E3-7}.

\subsubsection{Proof of Theorem \ref{T1-2}}
Let us take the same test function as shown in the proof of Theorem \ref{T1-1}. The function $\varphi\in C^{\infty}_{0}\left(\R^{N}\right)$ satisfying $0<\varphi\leq1$ in $B_{2}$ and
\begin{align*}
\varphi\left(x\right)=
\left\lbrace 
\begin{aligned}
&1,&&x\in B_{1},\\
&0,&&x\notin B_{2}.
\end{aligned}
\right.
\end{align*}
Obviously for any $i\in\mathbf{S}_{N}$, there exists $M>0$ such that 
\begin{align}\label{E3-8}
-\mathcal{I}_{i}\varphi\left(x\right)\leq M,\quad\forall x\in\R^{N}.
\end{align}
For any $m>\frac{p}{p-1}$ and $R>0$, we have
\begin{align*}
\int_{\R^{N}}u^{p}\varphi^{m}_{R}dx
\leq\int_{\R^{N}}\varphi^{m}_{R}\left(-\mathcal{I}u\right)dx
=\sum^{N}_{i=1}\int_{\R^{N}}\varphi^{m}_{R}\left(-\mathcal{I}_{i}u\right)dx.
\end{align*}
Let $x=\left(x_{i},x'_{i}\right)$. By Lemma \ref{L2-8}, for $u\in C^{2}\left(\R^{N}\right)\cap\mathcal{L}_{s}$, we can use Fubini's Theorem to obtain
\begin{align*}
\int_{\R^{N}}\varphi^{m}_{R}\left(-\mathcal{I}_{i}u\right)dx
=\int_{\R^{N-1}}\left(\int_{\R}\varphi^{m}_{R}\left(-\mathcal{I}_{i}u\right)dx_{i}\right)dx'_{i}.
\end{align*}
In virtue of Lemma \ref{L2-1}, Lemma \ref{L2-2}, Lemma \ref{L2-10} and \eqref{E3-8}, we infer that
\begin{align*}
\int_{\R}\varphi^{m}_{R}\left(-\mathcal{I}_{i}u\right)dx_{i}
&=\int_{\R}u\left(-\mathcal{I}_{i}\varphi^{m}_{R}\right)dx_{i}\\
&\leq m\int_{\R}u\varphi^{m-1}_{R}\left(-\mathcal{I}_{i}\varphi_{R}\right)dx_{i}\\
&\leq mMR^{-2s}\int_{\R}u\varphi^{m-1}_{R}dx_{i}.
\end{align*}
Therefore, we use Fubini's Theorem again to get 
\begin{align}\label{E3-9}
\begin{split}
\int_{\R^{N}}u^{p}\varphi^{m}_{R}dx
&\leq mMR^{-2s}\sum^{N}_{i=1}\int_{\R^{N-1}}\left(\int_{\R}u\varphi^{m-1}_{R}dx_{i}\right)dx'_{i}\\
&=mMR^{-2s}\sum^{N}_{i=1}\int_{\R^{N}}u\varphi^{m-1}_{R}dx\\
&=C\left(N,m,M\right)R^{-2s}\int_{\R^{N}}u\varphi^{m-1}_{R}dx.
\end{split}
\end{align}
By the H\"{o}lder inequality, we know
\begin{align*}
\int_{\R^{N}}u\varphi_{R}^{m-1}dx
\nonumber&\leq\left(\int_{B_{2R}}u^{p}\varphi_{R}^{m}dx\right)^{\frac{1}{p}}\left(\int_{B_{2R}}\varphi^{m-\frac{p}{p-1}}_{R}dx\right)^{\frac{p-1}{p}}\\
&\leq C\left(N,p\right)R^{\frac{N\left(p-1\right)}{p}}\left(\int_{\R^{N}}u^{p}\varphi_{R}^{m}dx\right)^{\frac{1}{p}},
\end{align*}
hence by \eqref{E3-9}, we finally conclude that
\begin{align}\label{E3-10}
\int_{\R^{N}}u^{p}\varphi_{R}^{m}dx\leq C\left(N,m,M,p\right)R^{N-\frac{2sp}{p-1}}.
\end{align}
The inequality \eqref{E3-10} derives Theorem \ref{T1-2} by the techniques employed in the proof of Theorem \ref{T1-1}.

\subsection{Symmetry of positive solutions: proof of Theorem \ref{T1-3} and \ref{T1-4}}\label{S3-2}
In this subsection, we will establish symmetry results. We use the equivalent formulation of $\mathcal{I}_{i}$ by
\begin{align*}
\mathcal{I}_{i}u\left(x\right)=\text{P.V.}\int_{\R}\frac{u\left(x+te_{i}\right)-u\left(x\right)}{\left|t\right|^{1+2s}}dt.
\end{align*}
The structure of the operator $\mathcal{I}$ imposes restrictions on the selection of directions, thus the conventional approach to proving that $u$ is radially symmetric cannot be fully applied. In this context, we will focus on $e_{1}$ direction and introduce the following notation: $\lambda\in\mathbb{R}$ and
\begin{align*}
x^{\lambda}=\left(2\lambda-x_{1},...,x_{N}\right);\quad u_{\lambda}\left(x\right)=u\left(x^{\lambda}\right);\quad w_{\lambda}\left(x\right)=u_{\lambda}\left(x\right)-u\left(x\right);
\end{align*}
\begin{align*}
T_{\lambda}=\left\{x\in\R^{N}:x_{1}=\lambda\right\};\quad\Sigma_{\lambda}=\left\{x\in\R^{N}:x_{1}<\lambda\right\};\quad\Sigma^{-}_{\lambda}=\left\{x\in\Sigma_{\lambda}:w_{\lambda}<0\right\}.
\end{align*}

\subsubsection{Proof of Theorem \ref{T1-3}-(1)}
Since $u\in C^{2}\left(\R^{N}\right)\cap\mathcal{L}_{s}$ is a positive solution of \eqref{E1-7}, then
\begin{align*}
-\mathcal{I}w_{\lambda}=u_{\lambda}^{p}-u^{p}\geq pu^{p-1}w_{\lambda}\quad\text{in }\Sigma^{-}_{\lambda}.
\end{align*}
Hence $w_{\lambda}\in C^{2}\left(\R^{N}\right)\cap\mathcal{L}_{s}$ satisfies 
\begin{align}\label{E3-11}
\left\lbrace 
\begin{aligned}
-\mathcal{I}w_{\lambda}-pu^{p-1}w_{\lambda}&\geq0&&\text{in }\Sigma^{-}_{\lambda},\\
w_{\lambda}&\geq0&&\text{in }\Sigma_{\lambda}\backslash\Sigma^{-}_{\lambda}.
\end{aligned}
\right.
\end{align}

\textbf{Step 1}. We first claim that for any $\lambda\in\R$, if $w_{\lambda}$ has a negative minimum point $\overline{x}$ in $\Sigma_{\lambda}$, then there exists $\rho=\rho\left(u\right)$ such that $\left|\overline{x}\right|\leq\rho$. Suppose that $\overline{x}$ is a negative minimum point of $w_{\lambda}$ in $\Sigma^{-}_{\lambda}$, by a similar calculation to get \eqref{E2-4}, we have
\begin{align}\label{E3-12}
-\mathcal{I}w_{\lambda}\left(\overline{x}\right)
\leq2w_{\lambda}\left(\overline{x}\right)\int^{+\infty}_{\lambda}\frac{1}{\left|t-\overline{x}_{1}\right|^{1+2s}}dt.
\end{align}
We assume by contradiction that $\left|\overline{x}\right|$ is very large ($\left|\overline{x}\right|>\lambda$). Let
\begin{align*}
\widetilde{x}=3\left|\overline{x}\right|e_{1}+\overline{x},
\end{align*}
then
\begin{align*}
\widetilde{x}_{1}-\left|\overline{x}\right|=2\left|\overline{x}\right|+\overline{x}_{1}\geq\left|\overline{x}\right|>\lambda.
\end{align*}
Thus we have
\begin{align}\label{E3-13}
\int_{\lambda}^{+\infty}\frac{1}{\left|t-\overline{x}_{1}\right|^{1+2s}}dt
>\int_{\widetilde{x}_{1}-\left|\overline{x}\right|}^{\widetilde{x}_{1}+\left|\overline{x}\right|}\frac{1}{\left|t-\overline{x}_{1}\right|^{1+2s}}dt
\geq C\left(N,s\right)\left|\overline{x}\right|^{-2s},
\end{align}
where we applied that
\begin{align*}
\left|t-\overline{x}_{1}\right|
\leq\left|t-\widetilde{x}_{1}\right|+\left|\widetilde{x}_{1}-\overline{x}_{1}\right|
\leq4\left|\widetilde{x}\right|.
\end{align*}
Consequently, we would get from \eqref{E3-12} and \eqref{E3-13} that
\begin{align*}
-\mathcal{I}w_{\lambda}\left(\overline{x}\right)
<C\left(N,s\right)w_{\lambda}\left(\overline{x}\right)\left|\overline{x}\right|^{-2s}.
\end{align*}
Let us see \eqref{E3-11} at the point $\overline{x}\in\Sigma^{-}_{\lambda}$, we obtain that
\begin{align*}
0&\leq-\mathcal{I}w_{\lambda}\left(\overline{x}\right)-pu^{p-1}\left(\overline{x}\right)w_{\lambda}\left(\overline{x}\right)\\
&<C\left(N,s\right)w_{\lambda}\left(\overline{x}\right)\left|\overline{x}\right|^{-2s}-pu^{p-1}\left(\overline{x}\right)w_{\lambda}\left(\overline{x}\right)\\
&=\left[C\left(N,s\right)\left|\overline{x}\right|^{-2s}-pu^{p-1}\left(\overline{x}\right)\right]w_{\lambda}\left(\overline{x}\right),
\end{align*}
which holds if and only if
\begin{align}\label{E3-14}
C\left(N,s\right)-pu^{p-1}\left(\overline{x}\right)\left|\overline{x}\right|^{2s}<0.
\end{align}
However, the condition \eqref{E1-9} indicates \eqref{E3-14} becomes impossible as $\left|\overline{x}\right|\rightarrow+\infty$. Thus, we have substantiated our claim.

\textbf{Step 2}. We assert that there exists $\rho_{0}$ such that for any $\lambda<-\rho_{0}$, there is $w_{\lambda}\geq0$ in $\Sigma_{\lambda}$. We assume $w_{\lambda}$ has a negative minimum point $\overline{x}$ in $\Sigma_{\lambda}$. According to the claim in \textbf{Step 1}, there exists $\rho_{0}$ that is independent of $\lambda$ such that $\overline{x}\in\Sigma_{\lambda}\cap\overline{B_{\rho_{0}}}$. However, for any $\lambda<-\rho_{0}$, we find that $\Sigma_{\lambda}\cap\overline{B_{\rho_{0}}}=\emptyset$, leading to a contradiction. Therefore, for any $\lambda<-\rho_{0}$, there cannot exist a negative minimum point of $w_{\lambda}$ in $\Sigma_{\lambda}$. Furthermore, based on the assumption stated in \eqref{E1-9}, we know 
\begin{align}\label{E3-15}
\lim_{\left|x\right|\rightarrow+\infty}w_{\lambda}\left(x\right)=0,
\end{align}
hence $w_{\lambda}$ does not assume any negative values in $\Sigma_{\lambda}$, and it follows that $w_{\lambda}\geq0$ in $\Sigma_{\lambda}$.

\textbf{Step 3}. We now move the plane $T_{\lambda}$ as $\lambda$ increases from $-\infty$ to $0$ and we define
\begin{align*}
\lambda_{1}:=\sup\left\{\lambda\leq0:w_{\eta}\geq0\text{ in }\Sigma_{\eta},\eta\leq\lambda\right\},
\end{align*}
as well as $\lambda$ decreases from $+\infty$ to $0$ and we define
\begin{align*}
\lambda_{2}:=\inf\left\{\lambda\geq0:w_{\eta}\leq0\text{ in }\Sigma_{\eta},\eta\geq\lambda\right\}.
\end{align*}

Case 1: $\lambda_{1}<0$ or $\lambda_{2}>0$. We only consider $\lambda_{1}<0$ and we will show that $w_{\lambda_{1}}\equiv0$ in $\Sigma_{\lambda_{1}}$. We assume by contradiction that 
\begin{align}\label{E3-16}
w_{\lambda_{1}}\geq0\quad\text{and}\quad w_{\lambda_{1}}\not\equiv0\quad\text{in }\Sigma_{\lambda_{1}}.
\end{align}
Under the assumption \eqref{E3-16}, if there exists $z\in\Sigma_{\lambda_{1}}$ such that $w_{\lambda_{1}}\left(z\right)=0$, then analogous to the calculations in the proof of Lemma \ref{L2-7}, when $i=1$, we obtain
\begin{align*}
\mathcal{I}_{1}w_{\lambda_{1}}\left(z\right)
&=\text{P.V.}\int_{\R}\frac{w_{\lambda_{1}}\left(z_{t}\right)}{\left|t-z_{1}\right|^{1+2s}}dt\\
&=\text{P.V.}\left(\int_{-\infty}^{\lambda_{1}}\frac{w_{\lambda_{1}}\left(z_{t}\right)}{\left|t-z_{1}\right|^{1+2s}}dt+\int_{\lambda_{1}}^{+\infty}\frac{w_{\lambda_{1}}\left(z_{t}\right)}{\left|t-z_{1}\right|^{1+2s}}dt\right)\\
&=\text{P.V.}\int^{\lambda_{1}}_{-\infty}\frac{w_{\lambda_{1}}\left(z_{t}\right)}{\left|t-z_{1}\right|^{1+2s}}+\frac{w_{\lambda_{1}}\left(z_{t}^{\lambda_{1}}\right)}{\left|2\lambda_{1}-t-z_{1}\right|^{1+2s}}dt\\
&=\text{P.V.}\int^{\lambda_{1}}_{-\infty}\left(\frac{1}{\left|t-z_{1}\right|^{1+2s}}-\frac{1}{\left|2\lambda_{1}-t-z_{1}\right|^{1+2s}}\right)w_{\lambda_{1}}\left(z_{t}\right)dt>0.
\end{align*}
When $i=2,...,N$, we have $z+te_{i}\in\Sigma_{\lambda_{1}}$, hence
\begin{align*}
\mathcal{I}_{i}w_{\lambda_{1}}\left(z\right)
=\text{P.V.}\int_{\R}\frac{w_{\lambda_{1}}\left(z+te_{i}\right)}{\left|t\right|^{1+2s}}dt\geq0.
\end{align*}
Therefore we get
\begin{align*}
0=u_{{\lambda}_{1}}^{p}\left(z\right)-u^{p}\left(z\right)=-\mathcal{I}w_{\lambda_{1}}\left(z\right)<0,
\end{align*}
this is not true. We conclude that if \eqref{E3-16} holds, then $w_{\lambda_{1}}>0$ in $\Sigma_{\lambda_{1}}$. 

For any given $\sigma>0$, in the bounded domain $\Sigma_{\lambda_{1}-\sigma}\cap B_{\rho_{0}}$, which is away from $T_{\lambda_{1}}$ and $T_{-\infty}$, there occurs
\begin{align*}
w_{\lambda_{1}}\geq C\left(\lambda_{1},\rho_{0},\sigma\right)\quad\text{in }\overline{\Sigma_{\lambda_{1}-\sigma}\cap B_{\rho_{0}}}.
\end{align*}
By the continuity of $w_{\lambda}$ about $\lambda$, there exists $\epsilon_{\sigma}\in\left(0,-\lambda_{1}\right)$ such that for any $\lambda\in\left[\lambda_{1},\lambda_{1}+\epsilon_{\sigma}\right)$,             
\begin{align*}
w_{\lambda}\geq0\quad\text{in }\overline{\Sigma_{\lambda_{1}-\sigma}\cap B_{\rho_{0}}}.
\end{align*}
By the claim in \textbf{Step 1}, we have
\begin{align*}
w_{\lambda}\geq0\quad\text{in }\Sigma_{\lambda_{1}-\sigma}\backslash\left(\Sigma_{\lambda_{1}-\sigma}\cap B_{\rho_{0}}\right).
\end{align*}
Consequently, for any $\lambda\in\left[\lambda_{1},\lambda_{1}+\epsilon_{\sigma}\right)$, there is
\begin{align}\label{E3-17}
w_{\lambda}\geq0\quad\text{in }\Sigma_{\lambda_{1}-\sigma}.
\end{align}
Let $D=\Sigma^{-}_{\lambda}\cap\left(\Sigma_{\lambda}\backslash\Sigma_{\lambda_{1}-\sigma}\right)$. We obtain that $w_{\lambda}\in C^{2}\left(D\right)\cap\mathcal{L}_{s}$ and satisfies
\begin{align}\label{E3-18}
\left\lbrace 
\begin{aligned}
-\mathcal{I}w_{\lambda}-pu^{p-1}w_{\lambda}&\geq0&&\text{in }D,\\
w_{\lambda}&\geq0&&\text{in }\Sigma_{\lambda}\backslash D.
\end{aligned}
\right.
\end{align}
We apply Lemma \ref{L2-7} with \eqref{E3-15} to \eqref{E3-18}, this leads to the conclusion that there exists small $\sigma_{0}>0$ such that for any $\sigma\in\left(0,\sigma_{0}\right)$, we have $w_{\lambda}\geq0$ in $D$. Let us choose at the begining that $\sigma\in\left(0,\sigma_{0}\right)$, then $D=\emptyset$, which implies
\begin{align}\label{E3-19}
w_{\lambda}\geq0\quad\text{in }\Sigma_{\lambda}\backslash\Sigma_{\lambda_{1}-\sigma}.
\end{align}
Combining \eqref{E3-17} and \eqref{E3-19}, we finally reach that for any $\lambda\in\left[\lambda_{1},\lambda_{1}+\epsilon_{\sigma}\right)$, it holds
\begin{align*}
w_{\lambda}\geq0\quad\text{in }\Sigma_{\lambda},
\end{align*}
this contradicts the definition of $\lambda_{1}$. Therefore $w_{\lambda_{1}}\equiv0$ in $\Sigma_{\lambda_{1}}$ and $u$ is symmetric about $T_{\lambda_{1}}$.

Case 2: $\lambda_{1}=\lambda_{2}=0$. In this case we obtain that $w_{\lambda_{1}}\geq0$ in $\Sigma_{\lambda_{1}}$ and $w_{\lambda_{2}}\leq0$ in $\Sigma_{\lambda_{2}}$, hence $w_{\lambda_{1}}=w_{\lambda_{2}}\equiv0$. Therefore $u$ is symmetric about $T_{0}$.

Through the aforementioned argument, we demonstrate that $u$ is symmetric about some $T_{\lambda}$ and $u$ decreases in $e_{1}$ direction. By considering all coordinate axis directions $e_{i}$, we can conclude that $u$ is symmetric with respect to $P_{e_{i},\widetilde{x}}$ about some $\widetilde{x}\in\R^{N}$ and $u$ decreases in $e_{i}$ direction.

\subsubsection{Proof of Theorem \ref{T1-3}-(2)}
We assume $u$ is radially symmetric about $\widetilde{x}$, then the function $v\left(x\right)=u\left(x+\widetilde{x}\right)=v\left(\left|x\right|\right)$ is also a solution of \eqref{E1-7}. Let us apply the Fourier transform to \eqref{E1-7}, we have by \eqref{E2-26} that 
\begin{align*}
\left|2\pi\xi\right|_{2s}^{2s}\widehat{v}\left(\left|\xi\right|\right)=\widehat{v^{p}}\left(\left|\xi\right|\right),
\end{align*}
that is
\begin{align*}
\left|2\pi\xi\right|_{2s}^{2s}=\frac{\widehat{v^{p}}\left(\left|\xi\right|\right)}{\widehat{v}\left(\left|\xi\right|\right)},
\end{align*}
we deduce that $\left|\xi\right|_{2s}$ is a radial symbol, which results in a contradiction.

\subsubsection{Proof of Theorem \ref{T1-4}-(1)}
Without loss of generality, we would consider $e_{1}$ direction.
\begin{lem}\label{L3-1}
When $s=\frac{1}{2}$, we have
\begin{align*}
G_{s}\left(x-y^{\lambda}\right)=G_{s}\left(x^{\lambda}-y\right),\quad\forall x,y\in\R^{N},x_{1}+y_{1}\neq 2\lambda
\end{align*}
and
\begin{align*}
G_{s}\left(x-y\right)>G_{s}\left(x^{\lambda}-y\right)>0,\quad\forall x,y\in\Sigma_{\lambda},x\neq y.
\end{align*}
\end{lem}
\begin{proof}
For $G_{s}$ defined by \eqref{E2-28}, let $\zeta=\left(-\xi_{1},\xi_{2},...,\xi_{N}\right)$, then
\begin{align*}
G_{s}\left(x-y^{\lambda}\right)
&=\frac{1}{\left(2\pi\right)^{N}}\int_{\R^{N}}\left|\xi\right|^{-2s}_{2s}\cos\left(\left(x-y^{\lambda}\right)\cdot\xi\right)d\xi\\
&=\frac{1}{\left(2\pi\right)^{N}}\int_{\R^{N}}\left|\zeta\right|^{-2s}_{2s}\cos\left(\left(x^{\lambda}-y\right)\cdot\zeta\right)d\zeta\\
&=G_{s}\left(x^{\lambda}-y\right).
\end{align*}
When $s=\frac{1}{2}$, we consider the form given in \eqref{E2-32} to obtain that for any $x,y\in\Sigma_{\lambda}$,
\begin{align*}
G_{s}\left(x-y\right)
&=\pi^{-N}\int^{+\infty}_{0}\left(\prod_{i=1}^{N}\frac{z}{z^{2}+\left|x_{i}-y_{i}\right|^{2}}\right)dz\\
&>\pi^{-N}\int^{+\infty}_{0}\left(\frac{z}{z^{2}+\left|2\lambda-x_{1}-y_{1}\right|^{2}}\cdot\prod_{i=2}^{N}\frac{z}{z^{2}+\left|x_{i}-y_{i}\right|^{2}}\right)dz\\
&=G_{s}\left(x^{\lambda}-y\right).
\end{align*}
\end{proof}
We start to prove Theorem \ref{T1-4}-(1). Since $u$ is a positive solution of \eqref{E1-10}, by Lemma \ref{L3-1}, we have
\begin{align*}
u\left(x\right)
&=\int_{\Sigma_{\lambda}}u^{p}\left(y\right)G_{s}\left(x-y\right)dy+\int_{\mathbb{R}^{N}\backslash\Sigma_{\lambda}}u^{p}\left(y\right)G_{s}\left(x-y\right)dy\\
&=\int_{\Sigma_{\lambda}}u^{p}\left(y\right)G_{s}\left(x-y\right)dy+\int_{\Sigma_{\lambda}}u_{\lambda}^{p}\left(y\right)G_{s}\left(x^{\lambda}-y\right)dy.
\end{align*}
Similarly, we would obtain
\begin{align*}
u_{\lambda}\left(x\right)
=\int_{\Sigma_{\lambda}}u^{p}\left(y\right)G_{s}\left(x^{\lambda}-y\right)dy+\int_{\Sigma_{\lambda}}u_{\lambda}^{p}\left(y\right)G_{s}\left(x-y\right)dy.
\end{align*}
we conclude that
\begin{align}\label{E3-20}
u\left(x\right)-u_{\lambda}\left(x\right)=\int_{\Sigma_{\lambda}}\left[G_{s}\left(x-y\right)-G_{s}\left(x^{\lambda}-y\right)\right]\left[u^{p}\left(y\right)-u_{\lambda}^{p}\left(y\right)\right]dy.
\end{align}

For any $x\in \Sigma_{\lambda}$, by Lemma \ref{L3-1} and \eqref{E3-20}, the Mean Value Theorem implies that
\begin{align*}
u\left(x\right)-u_{\lambda}\left(x\right)\leq\left(p-1\right)\int_{\Sigma^{-}_{\lambda}}G_{s}\left(x-y\right)u^{p-1}\left(y\right)\left[u\left(y\right)-u_{\lambda}\left(y\right)\right]dy.
\end{align*}
Since that $\left|G_{s}\right|\asymp\left|x\right|^{2s-N}$, hence the Hardy-Littlewood-Sobolev inequality and H\"{o}lder's inequality indicate that for any $\kappa>\frac{N}{N-2s}$, there is
\begin{align}\label{E3-21}
\begin{split}
\left\|w_{\lambda}\right\|_{L^{\kappa}\left({\Sigma^{-}_{\lambda}}\right)}
&\leq C\left(N,s,p,\kappa\right)\left\|u^{p-1}w_{\lambda}\right\|_{L^{\frac{N\kappa}{N+2s\kappa}}\left({\Sigma^{-}_{\lambda}}\right)}\\
&\leq C\left(N,s,p,\kappa\right)\left\|u\right\|^{p-1}_{L^{\frac{N\left(p-1\right)}{2s}}\left({\Sigma^{-}_{\lambda}}\right)}\left\|w_{\lambda}\right\|_{L^{\kappa}\left({\Sigma^{-}_{\lambda}}\right)}.
\end{split}
\end{align}
Since $u\in L^{\frac{N\left(p-1\right)}{2s}}\left(\R^{N}\right)$, there exists $\lambda_{0}<0$ such that for $\lambda\leq\lambda_{0}$, we have
\begin{align*}
C\left(N,s,p,\kappa\right)\left\|u\right\|^{p-1}_{L^{\frac{N\left(p-1\right)}{2s}}\left({\Sigma^{-}_{\lambda}}\right)}<1.
\end{align*}
Consequently we get by \eqref{E3-21} that
\begin{align*}
\left\|w_{\lambda}\right\|_{L^{\kappa}\left({\Sigma^{-}_{\lambda}}\right)}=0,
\end{align*}
that is $w_{\lambda}\geq0$ in $\Sigma_{\lambda}$. 

We move the plane $T_{\lambda}$ as $\lambda$ increases from $-\infty$ to $0$ and we define
\begin{align*}
\lambda_{1}:=\sup\left\{\lambda\leq0:w_{\eta}\geq0\text{ in }\Sigma_{\eta},\eta\leq\lambda\right\},
\end{align*}
as well as $\lambda$ decreases from $+\infty$ to $0$ and we define
\begin{align*}
\lambda_{2}:=\inf\left\{\lambda\geq0:w_{\eta}\leq0\text{ in }\Sigma_{\eta},\eta\geq\lambda\right\}.
\end{align*}

Case 1: $\lambda_{1}<0$ or $\lambda_{2}>0$. We only consider $\lambda_{1}<0$ and we will show that $w_{\lambda_{1}}\equiv0$ in $\Sigma_{\lambda_{1}}$. We assume by contradiction that 
\begin{align*}
w_{\lambda_{1}}\geq0\quad\text{and}\quad w_{\lambda_{1}}\not\equiv0\quad\text{in }\Sigma_{\lambda_{1}}.
\end{align*}
If there exists $z\in\Sigma_{\lambda_{1}}$ such that $w_{\lambda_{1}}\left(z\right)=0$, then by \eqref{E3-20}, we have $0=w_{{\lambda}_{1}}\left(z\right)<0$, this is not true, so we know $w_{\lambda_{1}}>0$ in $\Sigma_{\lambda_{1}}$.

Since $u\in L^{\frac{N\left(p-1\right)}{2s}}\left(\R^{N}\right)$, we can let $R>0$ such that 
\begin{align}\label{E3-22}
C\left(N,s,p,\kappa\right)\left\|u\right\|^{p-1}_{L^{\frac{N\left(p-1\right)}{2s}}\left(\R^{N}\backslash B_{R}\right)}<\frac{1}{2}.
\end{align}
For any $\tau>0$, we denote 
\begin{align*}
S_{\tau}=\left\{x\in\Sigma_{\lambda_{1}}:w_{\lambda_{1}}>\tau\right\}.
\end{align*}
Notice that $w_{\lambda_{1}}>0$ in $\Sigma_{\lambda_{1}}$, hence for any $\sigma>0$, there exists $\tau_{\sigma}>0$ such that
\begin{align}\label{E3-23}
\left|\left(\Sigma_{\lambda_{1}}\backslash S_{\tau_{\sigma}}\right)\cap B_{R}\right|<\sigma.
\end{align}
By the assumption $u\in L^{\frac{N\left(p-1\right)}{2s}}\left(\R^{N}\right)$, we know there exists $\epsilon_{\sigma}\in\left(0,-\lambda_{1}\right)$ such that for any $\lambda\in\left[\lambda_{1},\lambda_{1}+\epsilon_{\sigma}\right)$, we have
\begin{align}\label{E3-24}
\left|\left(\Sigma_{\lambda}\backslash\Sigma_{\lambda_{1}}\right)\cap B_{R}\right|<\sigma
\end{align}
and
\begin{align*}
\left|\left\{x\in B_{R}:u_{\lambda_{1}}-u_{\lambda}>\tau_{\sigma}\right\}\right|\leq\frac{1}{\tau_{\sigma}^{\frac{N\left(p-1\right)}{2s}}}\left\|u_{\lambda_{1}}-u_{\lambda}\right\|^{\frac{N\left(p-1\right)}{2s}}_{L^{\frac{N\left(p-1\right)}{2s}}\left({B_{R}}\right)}<\sigma.
\end{align*}
For any $x\in\Sigma^{-}_{\lambda}\cap S_{\tau_{\sigma}}\cap B_{R}$, we obtain
\begin{align*}
u_{\lambda_{1}}-u_{\lambda}=w_{\lambda_{1}}-w_{\lambda}>\tau_{\sigma},
\end{align*}
therefore
\begin{align}\label{E3-25}
\left|\Sigma^{-}_{\lambda}\cap S_{\tau_{\sigma}}\cap B_{R}\right|<\sigma.
\end{align}
We now follow the relationships
\begin{align*}
S_{\tau_{\sigma}}\cap\left(\Sigma_{\lambda_{1}}\backslash S_{\tau_{\sigma}}\right)\cap\left(\Sigma_{\lambda}\backslash\Sigma_{\lambda_{1}}\right)=\emptyset
\end{align*}
and
\begin{align*}
\Sigma^{-}_{\lambda}\subseteq\Sigma_{\lambda}=S_{\tau_{\sigma}}\cup\left(\Sigma_{\lambda_{1}}\backslash S_{\tau_{\sigma}}\right)\cup\left(\Sigma_{\lambda}\backslash\Sigma_{\lambda_{1}}\right),
\end{align*}
together with  \eqref{E3-23}, \eqref{E3-24} and \eqref{E3-25}, we find that
\begin{align*}
\left|\Sigma^{-}_{\lambda}\cap B_{R}\right|<\left|\Sigma^{-}_{\lambda}\cap S_{\tau_{\sigma}}\cap B_{R}\right|+\left|\left(\Sigma_{\lambda_{1}}\backslash S_{\tau_{\sigma}}\right)\cap B_{R}\right|+\left|\left(\Sigma_{\lambda}\backslash\Sigma_{\lambda_{1}}\right)\cap B_{R}\right|<3\sigma.
\end{align*}
Let us choose $\sigma$ to be sufficiently small such that
\begin{align*}
C\left(N,s,p,\kappa\right)\left\|u\right\|^{p-1}_{L^{\frac{N\left(p-1\right)}{2s}}\left(\Sigma^{-}_{\lambda}\cap B_{R}\right)}<\frac{1}{2}.
\end{align*}
Combining with \eqref{E3-22}, we have
\begin{align*}
C\left(N,s,p,\kappa\right)\left\|u\right\|^{p-1}_{L^{\frac{N\left(p-1\right)}{2s}}\left(\Sigma^{-}_{\lambda}\right)}<1.
\end{align*}
We get by \eqref{E3-21} that
\begin{align*}
\left\|w_{\lambda}\right\|_{L^{\kappa}\left({\Sigma^{-}_{\lambda}}\right)}=0,
\end{align*}
hence for any $\lambda\in\left[\lambda_{1},\lambda_{1}+\epsilon_{\sigma}\right)$, there holds
\begin{align*}
w_{\lambda}\geq0\quad\text{in }\Sigma_{\lambda},
\end{align*}
this contradicts the definition of $\lambda_{1}$. Therefore $w_{\lambda_{1}}\equiv0$ in $\Sigma_{\lambda_{1}}$ and $u$ is symmetric about $T_{\lambda_{1}}$.

Case 2: $\lambda_{1}=\lambda_{2}=0$. We obtain that $w_{\lambda_{1}}\geq0$ in $\Sigma_{\lambda_{1}}$ and $w_{\lambda_{2}}\leq0$ in $\Sigma_{\lambda_{2}}$, then we have $w_{\lambda_{1}}=w_{\lambda_{2}}\equiv0$. Therefore $u$ is symmetric about $T_{0}$.

We have already shown that $u$ is symmetric about some $T_{\lambda}$ and $u$ decreases in $e_{1}$ direction. Let us consider all the coordinate axis directions $e_{i}$, we obtain that $u$ is symmetric with respect to $P_{e_{i},\widetilde{x}}$ about some $\widetilde{x}\in\R^{N}$ and $u$ decreases in $e_{i}$ direction.

\subsubsection{Proof of Theorem \ref{T1-4}-(2)}
Without loss of generality, we consider $e_{1}+e_{2}$ direction. Since we are choosing the diagonal direction, we need to make slight modifications to the notation by
\begin{align*}
x^{\lambda}=\left(\lambda-x_{2},\lambda-x_{1},...,x_{N}\right);
\end{align*}
\begin{align*}
T_{\lambda}=\left\{x\in\R^{N}:x_{1}+x_{2}=\lambda\right\};\quad\Sigma_{\lambda}=\left\{x\in\R^{N}:x_{1}+x_{2}<\lambda\right\}.
\end{align*}

\begin{lem}\label{L3-2}
When $s=\frac{1}{2}$, we have
\begin{align*}
G_{s}\left(x-y^{\lambda}\right)=G_{s}\left(x^{\lambda}-y\right),\quad\forall x,y\in\R^{N},x_{1}+y_{1}\neq 2\lambda.
\end{align*}
and
\begin{align*}
G_{s}\left(x-y\right)>G_{s}\left(x^{\lambda}-y\right)>0,\quad\forall x,y\in\Sigma_{\lambda},x\neq y.
\end{align*}
\end{lem}
\begin{proof}
For $G_{s}$ defined by \eqref{E2-28}, let $\zeta=\left(-\xi_{2},-\xi_{1},\xi_{3},...,\xi_{N}\right)$, then
\begin{align*}
G_{s}\left(x-y^{\lambda}\right)
&=\frac{1}{\left(2\pi\right)^{N}}\int_{\R^{N}}\left|\xi\right|^{-2s}_{2s}\cos\left(\left(x-y^{\lambda}\right)\cdot\xi\right)d\xi\\
&=\frac{1}{\left(2\pi\right)^{N}}\int_{\R^{N}}\left|\zeta\right|^{-2s}_{2s}\cos\left(\left(x^{\lambda}-y\right)\cdot\zeta\right)d\zeta\\
&=G_{s}\left(x^{\lambda}-y\right).
\end{align*}
When $s=\frac{1}{2}$, by the form \eqref{E2-32}, for any $x,y\in\Sigma_{\lambda}$, we infer that
\begin{align*}
G_{\frac{1}{2}}\left(x-y\right)
&=\pi^{-N}\int^{+\infty}_{0}\left(\prod_{i=1}^{N}\frac{z}{z^{2}+\left|x_{i}-y_{i}\right|^{2}}\right)dz\\
&=\pi^{-N}\int^{+\infty}_{0}\left(\frac{z}{z^{2}+\left|x_{1}-y_{1}\right|^{2}}\cdot\frac{z}{z^{2}+\left|x_{2}-y_{2}\right|^{2}}\cdot\prod_{i=3}^{N}\frac{z}{z^{2}+\left|x_{i}-y_{i}\right|^{2}}\right)dz\\
&>\pi^{-N}\int^{+\infty}_{0}\left(\frac{z}{z^{2}+\left|\lambda-x_{2}-y_{1}\right|^{2}}\cdot\frac{z}{z^{2}+\left|\lambda-x_{1}-y_{2}\right|^{2}}\cdot\prod_{i=3}^{N}\frac{z}{z^{2}+\left|x_{i}-y_{i}\right|^{2}}\right)dz\\
&=G_{\frac{1}{2}}\left(x^{\lambda}-y\right),
\end{align*}
where we apply the formula
\begin{align*}
&\left(z^{2}+\left|\lambda-x_{2}-y_{1}\right|^{2}\right)\left(z^{2}+\left|\lambda-x_{2}-y_{1}\right|^{2}\right)
-\left(z^{2}+\left|x_{1}-y_{1}\right|^{2}\right)\left(z^{2}+\left|x_{2}-y_{2}\right|^{2}\right)\\
&=\left(\lambda-x_{1}-x_{2}\right)\left(\lambda-y_{1}-y_{2}\right)\left(2z^{2}+1\right)>0,\quad\forall\left(x,y,z\right)\in\Sigma_{\lambda}\times\Sigma_{\lambda}\times\R.
\end{align*}
\end{proof}

Utilizing Lemma \ref{L3-2}, we can proceed with the proof of Theorem \ref{T1-4}-(1). It follows that $u$ is symmetric about some $T_{\lambda_{0}}$ and $u$ decreases in $e_{1}+e_{2}$ direction. 

According to Theorem \ref{T1-4}-(1), we know $u$ is symmetric with respect to $P_{e_{i},\widetilde{x}}$ and $u$ decreases in $e_{i}$ direction. We will now assert that $\lambda_{0}=\widetilde{x}_{1}+\widetilde{x}_{2}$. To explore this assertion, we assume that $\lambda_{0}<\widetilde{x}_{1}+\widetilde{x}_{2}$, which implies that $\widetilde{x}^{\lambda_{0}}\in\Sigma_{\lambda_{0}}$. There would be
\begin{align*}
u\left(\widetilde{x}\right)
=u\left(\widetilde{x}^{\lambda_{0}}\right)
&=u\left(\lambda_{0}-\widetilde{x}_{2},\lambda_{0}-\widetilde{x}_{1},\widetilde{x}_{3},...,\widetilde{x}_{N}\right)\\
&<u\left(\widetilde{x}_{1},\lambda_{0}-\widetilde{x}_{1},\widetilde{x}_{3},...,\widetilde{x}_{N}\right)\\
&<u\left(\widetilde{x}_{1},\widetilde{x}_{2},\widetilde{x}_{3},...,\widetilde{x}_{N}\right)
=u\left(\widetilde{x}\right),
\end{align*}
this presents a contradiction. Concurrently, the condition $\lambda_{0} > \widetilde{x}_{1} + \widetilde{x}_{2}$ is also unattainable.

We consider all the directions $e_{i}\pm e_{j}$ to obtain that $u$ is symmetric with respect to $P_{e_{i}\pm e_{j},\widetilde{x}}$ and $u$ decreases in $e_{i}\pm e_{j}$ direction.

\subsubsection{Proof of Theorem \ref{T1-4}-(3)}
We assume $u$ is radially symmetric about $\widetilde{x}$, then $v\left(x\right)=u\left(x+\widetilde{x}\right)$ is also a radially symmetric solution of \eqref{E1-10}. We apply the Fourier transform to \eqref{E1-10}, we have
\begin{align*}
\widehat{v}\left(\left|\xi\right|\right)=\widehat{v^{p}}\left(\left|\xi\right|\right)\widehat{G_{s}}\left(\xi\right)=\widehat{v^{p}}\left(\left|\xi\right|\right)\left|\xi\right|_{2s}^{-2s},
\end{align*}
we deduce that $\left|\xi\right|_{2s}$ is a radial symbol, which leads to a contradiction.

\section{Nonexistence in the half space}\label{S4}

\subsection{Nonexistence of positive supersolutions: proof of Theorem \ref{T1-6}}

For any nonnegative solution $u\in C^{2}\left(\R_{+}^{N}\right)\cap\overline{\mathcal{L}_{s}}$ of \eqref{E1-11}, we can find a nonnegative solution of 
\begin{align}\label{E4-1}
\left\lbrace 
\begin{aligned}
-\mathcal{I}u&\geq u^{p}&&\text{in }\R_{+}^{N},\\
u&=0&&\text{in }\R^{N}\backslash\R_{+}^{N},
\end{aligned}
\right.
\end{align}
which satisfies \eqref{E2-13} and \eqref{E2-14}.
\begin{lem}\label{L4-1}
Assume $u\in C^{2}\left(\R_{+}^{N}\right)\cap\overline{\mathcal{L}_{s}}$ is a nonnegative solution of \eqref{E1-11}, let (Figure \ref{figure3})
\begin{align*}
\widetilde{u}\left(x\right)=
\left\lbrace 
\begin{aligned}
&u\left(x+e_{N}\right),&&x\in\R_{+}^{N},\\
&0,&&x\in\R^{N}\backslash\R_{+}^{N}.
\end{aligned}
\right.
\end{align*}
Then $u\in C^{2}\left(\R_{+}^{N}\right)\cap\overline{\mathcal{L}_{s}}$ that satisfies \eqref{E2-13} is a nonnegative solution of \eqref{E4-1}.

\begin{figure}[htbp]
\centering
\includegraphics[width=0.51\textwidth]{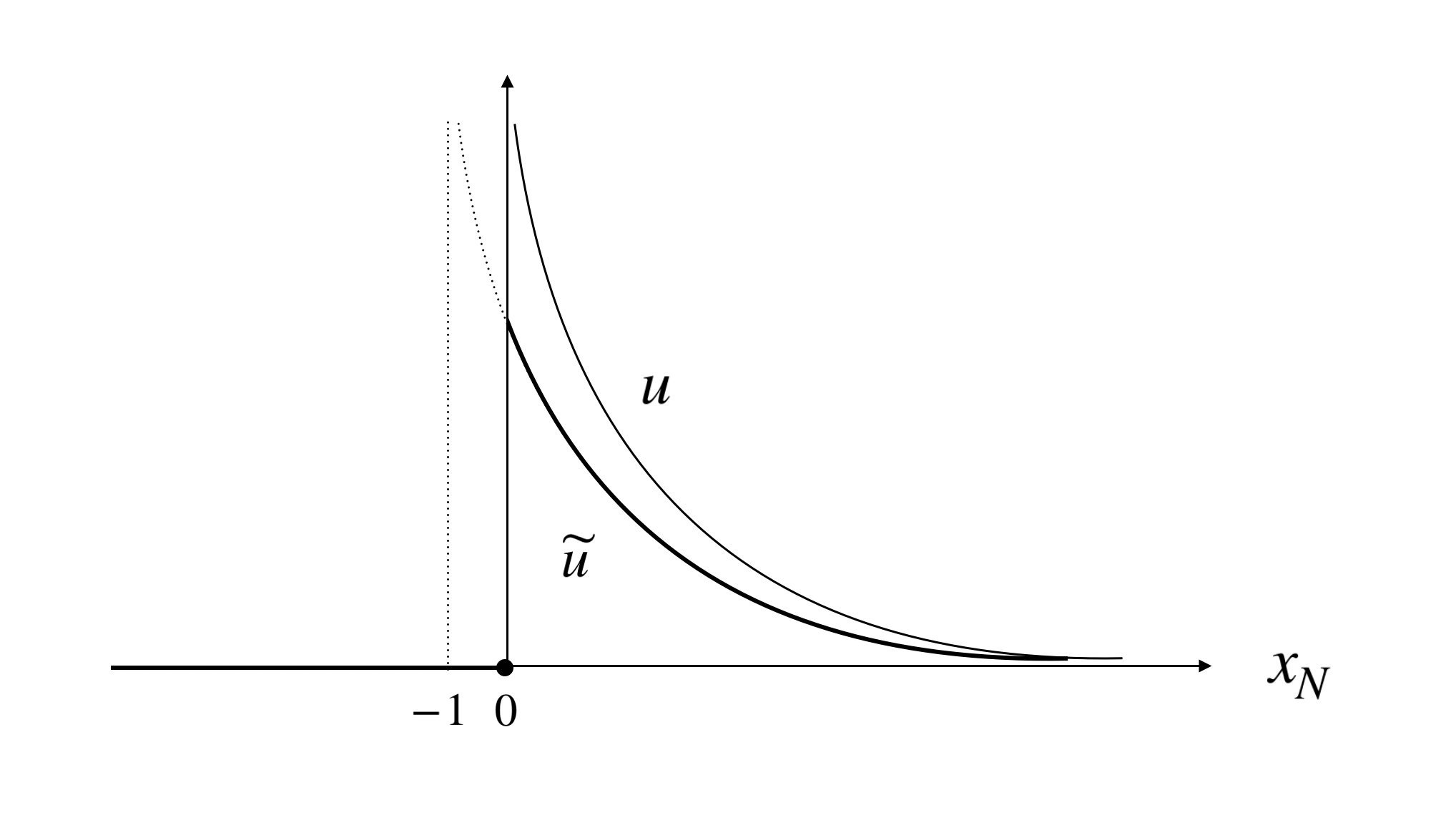}
\caption{The graph of $\widetilde{u}$.}
\label{figure3}
\end{figure}

\end{lem}

\begin{proof}
For any $\left(x,t\right)\in\R_{+}^{N}\times\R$, we have $\widetilde{u}\left(x\pm te_{i}\right)\leq u\left(x+e_{N}\pm te_{i}\right)$. Hence
\begin{align*}
-\mathcal{I}\widetilde{u}\left(x\right)
&=-\sum^{N}_{i=1}\int_{\R}\frac{\widetilde{u}\left(x+te_{i}\right)+\widetilde{u}\left(x-te_{i}\right)-2\widetilde{u}\left(x\right)}{\left|t\right|^{1+2s}}dt\\
&\geq-\sum^{N}_{i=1}\int_{\R}\frac{u\left(x+e_{N}+te_{i}\right)+u\left(x+e_{N}-te_{i}\right)-2u\left(x+e_{N}\right)}{\left|t\right|^{1+2s}}dt\\
&=-\mathcal{I}u\left(x+e_{N}\right).
\end{align*}
Since $u$ is a solution of \eqref{E1-11}, we have
\begin{align*}
-\mathcal{I}\widetilde{u}\left(x\right)\geq-\mathcal{I}u\left(x+e_{N}\right)\geq u^{p}\left(x+e_{N}\right)=\widetilde{u}^{p}\left(x\right).
\end{align*}
\end{proof}

After constructing a nonnegative solution $\widetilde{u}$ of \eqref{E4-1} in conjunction with \eqref{E2-13}, we can establish a Liouville theorem for the nonnegative solution of \eqref{E4-1}.

\begin{thm}\label{T4-2}
Let $0<s<1$ and $1<p\leq\frac{N+s}{N-s}$. Assume $u\in C^{2}\left(\R_{+}^{N}\right)\cap\overline{\mathcal{L}_{s}}$ is a nonnegative solution of \eqref{E4-1} and $u$ satisfies \eqref{E2-13}, then $u\equiv0$.
\end{thm}

\begin{proof}
In contrast to the selection of the standard test function utilized in the proof of Theorem \ref{T1-1}, we make a slight adjustment by elevating the test function in the direction of $e_{N}$. Let $\varphi\in C^{\infty}_{0}\left(\R^{N}\right)$ be  $0<\varphi\leq1$ in $B_{1}\left(\frac{e_{N}}{2}\right)$ and
\begin{align*}
\varphi\left(x\right)=
\left\lbrace 
\begin{aligned}
&1,&&x\in B_{\frac{3}{4}}\left(\frac{e_{N}}{2}\right),\\
&0,&&x\notin B_{1}\left(\frac{e_{N}}{2}\right).
\end{aligned}
\right.
\end{align*}
For any $s\leq\alpha<\min\left\{1,2s\right\}$, we define
\begin{align*}
\phi_{\alpha}\left(x\right)=\left(x_{N}\right)_{+}^{\alpha}\varphi\left(x\right).
\end{align*}

\textbf{Step 1.} For given $\alpha_{0}\in\left(s,\min\left\{1,2s\right\}\right)$, we show that there exists $M>0$ such that
\begin{align}\label{E4-2}
-\mathcal{I}\phi_{\alpha_{0}}\left(x\right)-\mathcal{I}\phi_{s}\left(x\right)\leq M\phi_{s}\left(x\right),\quad\forall x\in\R_{+}^{N}.
\end{align}
For any $x\in\R_{+}^{N}\cap\left\{\left|x-\frac{e_{N}}{2}\right|\geq1\right\}$, we have
\begin{align*}
-\mathcal{I}\phi_{\alpha_{0}}\left(x\right)-\mathcal{I}\phi_{s}\left(x\right)
&=-\sum_{i=1}^{N}\int_{\R}\frac{\phi_{\alpha_{0}}\left(x+te_{i}\right)+\phi_{\alpha_{0}}\left(x-te_{i}\right)}{\left|t\right|^{1+2s}}dt\\
&\quad\quad-\sum_{i=1}^{N}\int_{\R}\frac{\phi_{s}\left(x+te_{i}\right)+\phi_{s}\left(x-te_{i}\right)}{\left|t\right|^{1+2s}}dt\\
&\leq0=M\phi_{s}\left(x\right),
\end{align*}
therefore, in order to assert \eqref{E4-2}, it is necessary to establish that 
\begin{align}\label{E4-3}
\inf_{x\in\R_{+}^{N}\cap B_{1}\left(\frac{e_{N}}{2}\right)}\frac{\mathcal{I}\phi_{\alpha_{0}}\left(x\right)+\mathcal{I}\phi_{s}\left(x\right)}{\phi_{s}\left(x\right)}>-\infty.
\end{align}
We assume \eqref{E4-3} is not true, then there exists a convergent sequence $\left\{x_{n}\right\}\subset\left(\R_{+}^{N}\cap B_{1}\left(\frac{e_{N}}{2}\right)\right)$ such that
\begin{align}\label{E4-4}
\lim_{n\rightarrow+\infty}\frac{\mathcal{I}\phi_{\alpha_{0}}\left(x_{n}\right)+\mathcal{I}\phi_{s}\left(x_{n}\right)}{\phi_{s}\left(x_{n}\right)}=-\infty.
\end{align}
Let $x_{n}\rightarrow x_{\infty}\in\overline{\R_{+}^{N}\cap B_{1}\left(\frac{e_{N}}{2}\right)}$ as $n\rightarrow+\infty$. We have the fact that
\begin{align}\label{E4-5}
\left\lbrace 
\begin{aligned}
&\phi_{s}\left(x_{n}\right)>0,&&\forall n\in\N_{+},\\
&\phi_{s}\left(x_{\infty}\right)=0,&&\forall x_{\infty}\in\overline{\R_{+}^{N}}\cap\partial B_{1}\left(\frac{e_{N}}{2}\right).
\end{aligned}
\right.
\end{align}
We will consider the case where $x_{\infty}$ is positioned differently, which consistently leads to a contradiction with \eqref{E4-4}.

Case 1: $x_{\infty}\in\R_{+}^{N}\cap B_{1}\left(\frac{e_{N}}{2}\right)$. We obtain 
\begin{align*}
\lim_{n\rightarrow+\infty}\frac{\mathcal{I}\phi_{\alpha_{0}}\left(x_{n}\right)+\mathcal{I}\phi_{s}\left(x_{n}\right)}{\phi_{s}\left(x_{n}\right)}=\frac{\mathcal{I}\phi_{\alpha_{0}}\left(x_{\infty}\right)+\mathcal{I}\phi_{s}\left(x_{\infty}\right)}{\phi_{s}\left(x_{\infty}\right)},
\end{align*}
which contradicts \eqref{E4-4}.

Case 2: $x_{\infty}\in\R_{+}^{N}\cap \partial B_{1}\left(\frac{e_{N}}{2}\right)$. We see that
\begin{align*}
\mathcal{I}\phi_{\alpha_{0}}\left(x_{\infty}\right)+\mathcal{I}\phi_{s}\left(x_{\infty}\right)
&=\sum_{i=1}^{N}\int_{\R}\frac{\phi_{\alpha_{0}}\left(x_{\infty}+te_{i}\right)+\phi_{\alpha_{0}}\left(x_{\infty}-te_{i}\right)}{\left|t\right|^{1+2s}}dt\\
&\quad+\sum_{i=1}^{N}\int_{\R}\frac{\phi_{s}\left(x_{\infty}+te_{i}\right)+\phi_{s}\left(x_{\infty}-te_{i}\right)}{\left|t\right|^{1+2s}}dt\geq0.
\end{align*}
We assume $\mathcal{I}\phi_{\alpha_{0}}\left(x_{\infty}\right)+\mathcal{I}\phi_{s}\left(x_{\infty}\right)=0$, then we have
\begin{align*}
\phi_{\alpha_{0}}\left(x_{\infty}\pm te_{i}\right)=\phi_{s}\left(x_{\infty}\pm te_{i}\right)=0,\quad\forall i\in\mathbf{S}_{N},\forall t\in\R,
\end{align*}
which is equivalent to
\begin{align*}
\left\lbrace 
\begin{aligned}
&\varphi\left(x_{\infty}\pm te_{i}\right)=0,&&\forall i\in\mathbf{S}_{N-1},\forall t\in\R,\\
&\left(\left(x_{\infty}\right)_{N}\pm t\right)_{+}\cdot\varphi\left(x_{\infty}\pm te_{N}\right)=0,&&\forall t\in\R.
\end{aligned}
\right.
\end{align*}
Now that $x_{\infty}\in\R_{+}^{N}\cap\partial B_{1}\left(\frac{e_{N}}{2}\right)$, the above condition cannot exist, thus we would obtain that
\begin{align*}
\lim_{n\rightarrow+\infty}\left[\mathcal{I}\phi_{\alpha_{0}}\left(x_{n}\right)+\mathcal{I}\phi_{s}\left(x_{n}\right)\right]
=\mathcal{I}\phi_{\alpha_{0}}\left(x_{\infty}\right)+\mathcal{I}\phi_{s}\left(x_{\infty}\right)>0.
\end{align*}
Therefore, we get by \eqref{E4-5} that
\begin{align*}
\lim_{n\rightarrow+\infty}\frac{\mathcal{I}\phi_{\alpha_{0}}\left(x_{n}\right)+\mathcal{I}\phi_{s}\left(x_{n}\right)}{\phi_{s}\left(x_{n}\right)}=+\infty.
\end{align*}

Case 3: $x_{\infty}\in\partial\R_{+}^{N}\cap B_{1}\left(\frac{e_{N}}{2}\right)$. By Lemma \ref{L2-3}, Lemma \ref{L2-4} and Lemma \ref{L2-5}, we have for any $x\in \R^{N}_{+}$,
\begin{align}\label{E4-6}
\begin{split}
\mathcal{I}\phi_{\alpha_{0}}\left(x\right)+\mathcal{I}\phi_{s}\left(x\right)
=&C_{\alpha_{0}}\varphi\left(x\right)x^{\alpha_{0}-2s}_{N}+x_{N}^{\alpha_{0}}\mathcal{I}\varphi\left(x\right)+x_{N}^{s}\mathcal{I}\varphi\left(x\right)\\
&+I_{N}\left[\omega_{\alpha_{0}},\varphi\right]\left(x\right)+I_{N}\left[\omega_{s},\varphi\right]\left(x\right),
\end{split}
\end{align}
where $C_{\alpha_{0}}>0$ and 
\begin{align*}
&I_{N}\left[\omega_{\alpha_{0}},\varphi\right]\left(x\right)+I_{N}\left[\omega_{s},\varphi\right]\left(x\right)\\
&=\int_{\R}\frac{\left[\left(x_{N}+t\right)_{+}^{\alpha_{0}}+\left(x_{N}+t\right)_{+}^{s}-x_{N}^{\alpha_{0}}-x_{N}^{s}\right]\left[\varphi\left(x+te_{N}\right)-\varphi\left(x\right)\right]}{\left|t\right|^{1+2s}}\\
&\quad+\frac{\left[\left(x_{N}-t\right)_{+}^{\alpha_{0}}+\left(x_{N}-t\right)_{+}^{s}-x_{N}^{\alpha_{0}}-x_{N}^{s}\right]\left[\varphi\left(x-te_{N}\right)-\varphi\left(x\right)\right]}{\left|t\right|^{1+2s}}dt.
\end{align*}
It is straightforward to acquire that
\begin{align*}
\lim_{n\rightarrow+\infty}C_{\alpha_{0}}\varphi\left(x_{n}\right)\cdot\left(x_{n}\right)_{N}^{\alpha_{0}-2s}=+\infty
\end{align*}
and
\begin{align*}
\lim_{n\rightarrow+\infty}\left[\left(x_{n}\right)_{N}^{\alpha_{0}}\mathcal{I}\varphi\left(x_{n}\right)+\left(x_{n}\right)_{N}^{s}\mathcal{I}\varphi\left(x_{n}\right)\right]=0.
\end{align*}
By Lemma \ref{L2-5}, we obtain
\begin{align*}
\lim_{n\rightarrow+\infty}\left[I_{N}\left[\omega_{\alpha_{0}},\varphi\right]\left(x_{n}\right)+I_{N}\left[\omega_{s},\varphi\right]\left(x_{n}\right)\right]=I_{N}\left[\omega_{\alpha_{0}},\varphi\right]\left(x_{\infty}\right)+I_{N}\left[\omega_{s},\varphi\right]\left(x_{\infty}\right),
\end{align*}
hence \eqref{E4-6} yields
\begin{align*}
\lim_{n\rightarrow+\infty}\left[\mathcal{I}\phi_{\alpha_{0}}\left(x_{n}\right)+\mathcal{I}\phi_{s}\left(x_{n}\right)\right]=+\infty.
\end{align*}
We reach by \eqref{E4-5} that
\begin{align*}
\lim_{n\rightarrow+\infty}\frac{\mathcal{I}\phi_{\alpha_{0}}\left(x_{n}\right)+\mathcal{I}\phi_{s}\left(x_{n}\right)}{\phi_{s}\left(x_{n}\right)}=+\infty.
\end{align*}

Case 4: $x_{\infty}\in\partial\R_{+}^{N}\cap\partial B_{1}\left(\frac{e_{N}}{2}\right)$. Note that by \eqref{E4-5} and \eqref{E4-6}, we have
\begin{align}\label{E4-7}
\begin{split}
\mathcal{I}\phi_{\alpha_{0}}\left(x_{n}\right)+\mathcal{I}\phi_{s}\left(x_{n}\right)
\geq\left(x_{n}\right)_{N}^{\alpha_{0}}\mathcal{I}\varphi\left(x_{n}\right)+\left(x_{n}\right)_{N}^{s}\mathcal{I}\varphi\left(x_{n}\right)
+I_{N}\left[\omega_{\alpha_{0}},\varphi\right]\left(x_{n}\right)+I_{N}\left[\omega_{s},\varphi\right]\left(x_{n}\right).
\end{split}
\end{align}
It is obvious that
\begin{align}\label{E4-8}
\lim_{n\rightarrow+\infty}\left[\left(x_{n}\right)_{N}^{\alpha_{0}}\mathcal{I}\varphi\left(x_{n}\right)+\left(x_{n}\right)_{N}^{s}\mathcal{I}\varphi\left(x_{n}\right)\right]=0.
\end{align}
We use Lemma \ref{L2-5} to know (See Figure \ref{figure4})
\begin{align}\label{E4-9}
\begin{split}
&\lim_{n\rightarrow+\infty}\left[I_{N}\left[\omega_{\alpha_{0}},\varphi\right]\left(x_{n}\right)+I_{N}\left[\omega_{s},\varphi\right]\left(x_{n}\right)\right]\\
&=\int_{\R}\frac{\left(t_{+}^{\alpha_{0}}+t_{+}^{s}\right)\varphi\left(x_{\infty}+te_{N}\right)+\left(\left(-t\right)_{+}^{\alpha_{0}}+\left(-t\right)_{+}^{s}\right)\varphi\left(x_{\infty}-te_{N}\right)}{\left|t\right|^{1+2s}}dt>0.
\end{split}
\end{align}
By \eqref{E4-5}, \eqref{E4-7}, \eqref{E4-8} and \eqref{E4-9}, we conclude that
\begin{align*}
\lim_{n\rightarrow+\infty}\frac{\mathcal{I}\phi_{\alpha_{0}}\left(x_{n}\right)+\mathcal{I}\phi_{s}\left(x_{n}\right)}{\phi_{s}\left(x_{n}\right)}=+\infty.
\end{align*}

\begin{figure}[htbp]
\centering
\includegraphics[width=0.55\textwidth]{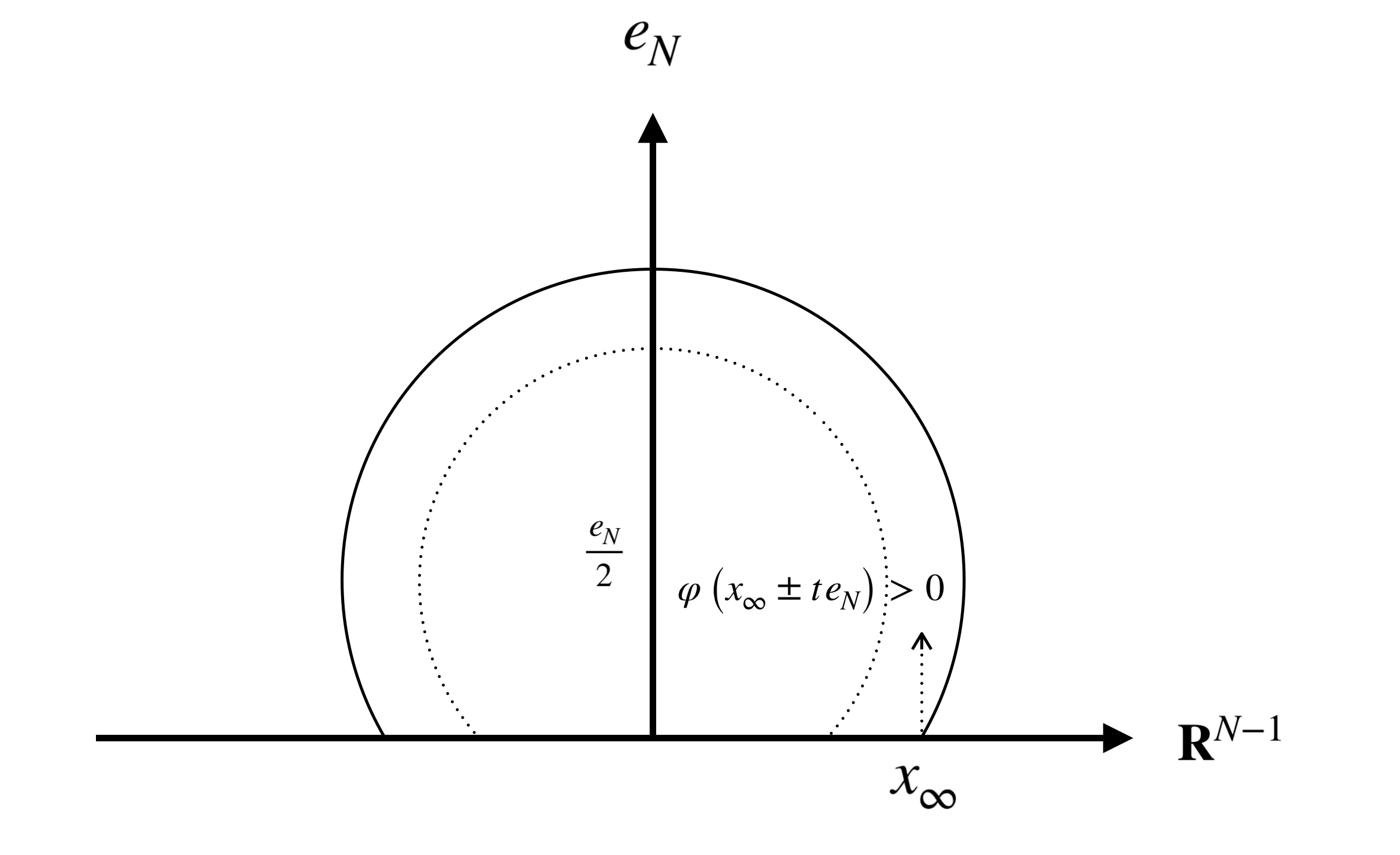}
\caption{$\varphi\left(x_{\infty}\pm te_{N}\right)$ does not always be $0$ after moving up the ball.}
\label{figure4}
\end{figure}

In conclusion, for any $x_{\infty}\in\overline{\R_{+}^{N}\cap B_{1}\left(\frac{e_{N}}{2}\right)}$, a contradiction to \eqref{E4-4} is consistently observed, hence then inequality \eqref{E4-2} is proved to be true.

\textbf{Step 2.} For any $R>0$, we perform the following rescaling
\begin{align*}
\varphi_{R}\left(x\right)=\varphi\left(\frac{x}{R}\right)=
\left\lbrace 
\begin{aligned}
&1,&&x\in B_{\frac{3R}{4}}\left(\frac{Re_{N}}{2}\right),\\
&0,&&x\notin B_{R}\left(\frac{Re_{N}}{2}\right).
\end{aligned}
\right.
\end{align*}
For any $s\leq\alpha<\min\left\{1,2s\right\}$, we define
\begin{align*}
\phi_{\alpha,R}\left(x\right)=\left(x_{N}\right)_{+}^{\alpha}\varphi_{R}\left(x\right).
\end{align*}
Since $u\in C^{2}\left(\R_{+}^{N}\right)\cap\overline{\mathcal{L}_{s}}$ is a nonnegative solution of \eqref{E4-1} and $u$ satisfies \eqref{E2-13}, by Lemma \ref{L2-1} and Lemma \ref{L2-11}, we have
\begin{align*}
\int_{\R_{+}^{N}}u^{p}\phi_{s,R}dx
&\leq R^{s}\int_{\R^{N}}\left(\frac{x_{N}}{R}\right)_{+}^{s}\varphi_{R}\left(-\mathcal{I}u\right)dx\\
&\leq R^{s}\int_{\R^{N}}\left(\frac{x_{N}}{R}\right)_{+}^{s}\varphi_{R}\left(-\mathcal{I}u\right)dx+R^{s}\int_{\R^{N}}\left(\frac{x_{N}}{R}\right)_{+}^{\alpha_{0}}\varphi_{R}\left(-\mathcal{I}u\right)dx\\
&=R^{-s}\int_{\R^{N}}u\left[-\mathcal{I}\phi_{s}\left(\frac{x}{R}\right)\right]dx+R^{-s}\int_{\R^{N}}u\left[-\mathcal{I}\phi_{\alpha_{0}}\left(\frac{x}{R}\right)\right]dx\\
&=R^{-s}\int_{\R_{+}^{N}}u\left[-\mathcal{I}\phi_{\alpha_{0}}\left(\frac{x}{R}\right)-\mathcal{I}\phi_{s}\left(\frac{x}{R}\right)\right]dx.
\end{align*}
We deduce by \eqref{E4-2} that
\begin{align*}
\int_{\R_{+}^{N}}u\left[-\mathcal{I}\phi_{\alpha_{0}}\left(\frac{x}{R}\right)-\mathcal{I}\phi_{s}\left(\frac{x}{R}\right)\right]dx
\leq M\int_{\R_{+}^{N}}u\phi_{s}\left(\frac{x}{R}\right)dx
=MR^{-s}\int_{\R_{+}^{N}}u\phi_{s,R}dx.
\end{align*}
We can reach that
\begin{align}\label{E4-10}
\int_{\R_{+}^{N}}u^{p}\phi_{s,R}dx\leq MR^{-2s}\int_{\R_{+}^{N}}u\phi_{s,R}dx.
\end{align}
Since that for any $x\in\R_{+}^{N}$, there holds
\begin{align}\label{E4-11}
0\leq\phi_{s,R}\left(x\right)=\left(x_{N}\right)_{+}^{s}\varphi_{R}\left(x\right)\leq\left(x_{N}\right)_{+}^{s},
\end{align}
hence by the H\"{o}lder inequality,
\begin{align*}
\int_{\R_{+}^{N}}u\phi_{s,R}dx
&=\int_{\R_{+}^{N}\cap B_{2R}}u\phi_{s,R}dx\\
&\leq\left(\int_{\R_{+}^{N}\cap B_{2R}}u^{p}\phi_{s,R}dx\right)^{\frac{1}{p}}\left(\int_{\R_{+}^{N}\cap B_{2R}}\phi_{s,R}dx\right)^{\frac{p-1}{p}}\\
&\leq C\left(N,p\right)R^{\frac{\left(N+s\right)\left(p-1\right)}{p}}\left(\int_{\R_{+}^{N}}u^{p}\phi_{s,R}dx\right)^{\frac{1}{p}}.
\end{align*}
We infer by \eqref{E4-10} that for any $p>1$, 
\begin{align}\label{E4-12}
\int_{\R_{+}^{N}}u^{p}\phi_{s,R}dx\leq C\left(N,M,p\right)R^{N+s-\frac{2sp}{p-1}}.
\end{align}

\textbf{Step 3.} When $R\rightarrow+\infty$, there occurs 
\begin{align*}
\left\lbrace 
\begin{aligned}
\varphi_{R}&\rightarrow\varphi_{\infty}\equiv1&&\text{in }\R^{N}_{+},\\
\phi_{s,R}&\rightarrow\phi_{s,\infty}=x_{N}^{s}&&\text{in }\R^{N}_{+}.
\end{aligned}
\right.
\end{align*}
When $1<p<\frac{N+s}{N-s}$, we let $R\rightarrow+\infty$ in \eqref{E4-12}, we get $u\equiv0$. When $p=\frac{N+s}{N-s}$, as $R\rightarrow+\infty$, we have
\begin{align}\label{E4-13}
\int_{\R_{+}^{N}}x^{s}_{N}u^{p}dx<+\infty.
\end{align}
We rewrite
\begin{align*}
\int_{\R_{+}^{N}}u\phi_{s,R}dx
=\int_{\R_{+}^{N}\cap\left\{\left|x\right|\leq\sqrt{R}\right\}}u\phi_{s,R}dx+\int_{\R_{+}^{N}\cap\left\{\sqrt{R}\leq\left|x\right|\leq 2R\right\}}u\phi_{s,R}dx.
\end{align*}
By \eqref{E4-11}, the H\"{o}lder inequality implies that
\begin{align*}
\int_{\R_{+}^{N}\cap\left\{\left|x\right|\leq\sqrt{R}\right\}}u\phi_{s,R}dx
&\leq\left(\int_{\R_{+}^{N}\cap\left\{\left|x\right|\leq\sqrt{R}\right\}}x_{N}^{s}u^{p}dx\right)^{\frac{1}{p}}\left(\int_{\R_{+}^{N}\cap\left\{\left|x\right|\leq\sqrt{R}\right\}}x_{N}^{s}dx\right)^{\frac{p-1}{p}}\\
&\leq C\left(N,p\right)R^{\frac{\left(N+s\right)\left(p-1\right)}{2p}}\left(\int_{\R_{+}^{N}}x^{s}_{N}u^{p}dx\right)^{\frac{1}{p}}\\
&=C\left(N,p\right)R^{s}\left(\int_{\R_{+}^{N}}x^{s}_{N}u^{p}dx\right)^{\frac{1}{p}}
\end{align*}
and
\begin{align*}
\int_{\R_{+}^{N}\cap\left\{\sqrt{R}\leq\left|x\right|\leq 2R\right\}}u\phi_{s,R}dx
&\leq\left(\int_{\R_{+}^{N}\cap\left\{\sqrt{R}\leq\left|x\right|\leq 2R\right\}}x_{N}^{s}u^{p}dx\right)^{\frac{1}{p}}\left(\int_{\R_{+}^{N}\cap\left\{\sqrt{R}\leq\left|x\right|\leq 2R\right\}}x_{N}^{s}dx\right)^{\frac{p-1}{p}}\\
&\leq C\left(N,p\right)R^{\frac{\left(N+s\right)\left(p-1\right)}{p}}\left(\int_{\R_{+}^{N}\cap\left\{\sqrt{R}\leq\left|x\right|\leq 2R\right\}}x^{s}_{N}u^{p}dx\right)^{\frac{1}{p}}\\
&=C\left(N,p\right)R^{2s}\left(\int_{\R_{+}^{N}\cap\left\{\sqrt{R}\leq\left|x\right|\leq 2R\right\}}x^{s}_{N}u^{p}dx\right)^{\frac{1}{p}}.
\end{align*}
By \eqref{E4-10}, we conclude that
\begin{align*}
\int_{\R_{+}^{N}}u^{p}\phi_{s,R}dx\leq C\left(N,M,p\right)\left[R^{-s}\left(\int_{\R_{+}^{N}}x^{s}_{N}u^{p}dx\right)^{\frac{1}{p}}+\left(\int_{\R_{+}^{N}\cap\left\{\sqrt{R}\leq\left|x\right|\leq 2R\right\}}x^{s}_{N}u^{p}dx\right)^{\frac{1}{p}}\right].
\end{align*}
Therefore, in virtue of \eqref{E4-13}, we reach $u\equiv0$ by letting $R\rightarrow+\infty$.
\end{proof}

\subsubsection{Proof of Theorem \ref{T1-6}}
Let $u\in C^{2}\left(\R_{+}^{N}\right)\cap\overline{\mathcal{L}_{s}}$ be any nonnegative solution of \eqref{E1-11}. According to Lemma \ref{L4-1}, the function $\widetilde{u}\in C^{2}\left(\R_{+}^{N}\right)\cap\overline{\mathcal{L}_{s}}$, as defined in Lemma \ref{L4-1}, is a nonnegative solution of \eqref{E4-1} and  $\widetilde{u}$ satisfies \eqref{E2-13}. Theorem \ref{T4-2} implies that $\widetilde{u}\equiv0$ in $\R^N$, hence the continuity of $u$ in $\R^N$ shows that
\begin{align*}
u
\left\lbrace 
\begin{aligned}
&\geq0,&&x_{N}<1,\\
&=0,&&x_{N}\geq1.
\end{aligned}
\right.
\end{align*}
For any $\overline{x}$ with $\overline{x}_{N}=1$, we have
\begin{align*}
0=u^{p}\left(\overline{x}\right)\leq-\mathcal{I}u\left(\overline{x}\right)=-\sum_{i=1}^{N}\int_{\R}\frac{u\left(\overline{x}+te_{i}\right)+u\left(\overline{x}-te_{i}\right)}{\left|t\right|^{1+2s}}dt\leq0,
\end{align*}
that is $u\left(\overline{x}\pm te_{i}\right)=0$ holds a.e. for any $i\in\mathbf{S}_{N}$ and any $t\in\R$ (i.e. $u=0$ a.e. in $A\left(\overline{x}\right)$). Since $u\in C^{2}\left(\R_{+}^{N}\right)$, the arbitrariness of $\overline{x}$ leads to the conclusion that $u\equiv0$ in $\R^{N}_{+}$ and $u=0$ a.e. in $\R^{N}\backslash\R_{+}^{N}$.

\subsection{Nonexistence of positive solutions: proof of Theorem \ref{T1-7}}
We choose the direction $e_{N}$, all the notation in Section \ref{S3-2} should be transferred to the version of $e_{N}$ direction. Let $\lambda>0$ and we define
\begin{align*}
\Sigma'_{\lambda}=\left\{x\in\R_{+}^{N}:0<x_{N}<\lambda\right\};\quad\Sigma'^{-}_{\lambda}=\left\{x\in\Sigma'_{\lambda}:w_{\lambda}<0\right\}.
\end{align*}
Since $u\in C^{2}\left(\R_{+}^{N}\right)\cap\mathcal{L}_{s}$ is a nonnegative solution of \eqref{E1-12} and $u$ satisfies \eqref{E1-9}, then Lemma \ref{L2-6} with $\Omega=\R_{+}^{N}$ indicates that either $u>0$ or $u\equiv0$ in $\R_{+}^{N}$. Therefore, we only need to exclude the case $u>0$ in $\R_{+}^{N}$. 

We assume $u>0$ in $\R_{+}^{N}$. We can verify that $w_{\lambda}\in C^{2}\left(\Sigma'^{-}_{\lambda}\right)\cap\mathcal{L}_{s}$ satisfies 
\begin{align}\label{E4-14}
\left\lbrace 
\begin{aligned}
-\mathcal{I}w_{\lambda}-pu^{p-1}w_{\lambda}&\geq0&&\text{in }\Sigma'^{-}_{\lambda},\\
w_{\lambda}&\geq0&&\text{in }\Sigma_{\lambda}\backslash\Sigma'^{-}_{\lambda}.
\end{aligned}
\right.
\end{align}
We apply Lemma \ref{L2-7} with $D=\Sigma'^{-}_{\lambda}$ and \eqref{E1-9} to \eqref{E4-14}, we obtain that there exists small $\lambda_{0}>0$ such that for any $\lambda\in\left(0,\lambda_{0}\right)$, we have $w_{\lambda}\geq0$ in $\Sigma'^{-}_{\lambda}$. Hence $\Sigma'^{-}_{\lambda}=\emptyset$, which implies
\begin{align*}
w_{\lambda}\geq0\quad\text{in }\Sigma'_{\lambda}.
\end{align*}

We now move the plane $T_{\lambda}$ as $\lambda$ increases from $0$ to $+\infty$ and we define
\begin{align*}
\lambda_{1}:=\sup\left\{\lambda>0:w_{\eta}\geq0\text{ in }\Sigma'_{\eta},\eta\leq\lambda\right\}.
\end{align*}

If $\lambda_{1}=+\infty$, then $u$ increases in $e_{N}$ direction, this situation contradicts the conditions $u>0$ in $\R_{+}^{N}$ and $u$ satisfies \eqref{E1-9}. If $0<\lambda_{1}<+\infty$, we can employ a similar argument that presented in the proof of Theorem \ref{T1-3}-(1) to show that $w_{\lambda_{1}}\equiv0$ in $\Sigma'_{\lambda_{1}}$. For any $x\in\R_{+}^{N}$ with $x_{N}>2\lambda_{1}$, we have $x^{\lambda_{1}}\in\R^{N}\backslash\R_{+}^{N}$ and
\begin{align*}
u\left(x\right)=u\left(x^{\lambda_{1}}\right)=0.
\end{align*}
This represents yet another contradiction. Therefore, there only holds $u\equiv0$ in $\R_{+}^{N}$.

\section{Symmetry in the unit ball}\label{S5}
\subsection{Symmetry of positive solutions: proof of Theorem \ref{T1-8}}
Without loss of generality, we begin by selecting $e_{1}$ direction. We will still use the notation in Section \ref{S3-2} and proceed to define a new domain 
\begin{align*}
\Sigma''_{\lambda}=\left\{x\in B_{1}:-1<x_{1}<\lambda,\lambda\in\left(-1,1\right)\right\}.
\end{align*}
Notice that $u\in C^{2}\left(B_{1}\right)\cap\mathcal{L}_{s}$ is a solution of \eqref{E1-13} and $f$ is a Lipschitz continuous function, we can ascertain that
\begin{align*}
-\mathcal{I}w_{\lambda}=\frac{f\left(u_{\lambda}\right)-f\left(u\right)}{u_{\lambda}-u}w_{\lambda}\geq C\left(f\right)w_{\lambda}\quad\text{in }B_{1}.
\end{align*}
Hence $w_{\lambda}\in C^{2}\left(\Sigma''_{\lambda}\right)\cap\mathcal{L}_{s}$ satisfies  
\begin{align}\label{E5-1}
\left\lbrace 
\begin{aligned}
-\mathcal{I}w_{\lambda}-C\left(f\right)w_{\lambda}&\geq0&&\text{in }\Sigma''_{\lambda},\\
w_{\lambda}&\geq0&&\text{in }\Sigma_{\lambda}\backslash\Sigma''_{\lambda}.
\end{aligned}
\right.
\end{align}

Let $D=\Sigma''_{\lambda}$ and we apply Lemma \ref{L2-7} to \eqref{E5-1}, we obtain that there exists $\lambda_{0}\in\left(-1,0\right)$ such that for any $\lambda\in\left(-1,\lambda_{0}\right)$, we have $w_{\lambda}\geq0$ in $\Sigma''_{\lambda}$. We move the plane $T_{\lambda}$ as $\lambda$ increases from $-1$ to $0$ and we define
\begin{align*}
\lambda_{1}:=\sup\left\{\lambda\leq0:w_{\eta}\geq0\text{ in }\Sigma''_{\eta},\eta\leq\lambda\right\}.
\end{align*}

We assert that $\lambda_{1}=0$. To establish this, we proceed by contradiction, assuming that $-1<\lambda_{1}<0$. Under this assumption, it follows that $w_{\lambda_{1}}\geq0$ in $\Sigma''_{\lambda_{1}}$. Consequently, we will obtain that $w_{\lambda_{1}}>0$ in $\Sigma''_{\lambda_{1}}$. Alternatively, if there exists $z\in\Sigma''_{\lambda_{1}}$ such that $w_{\lambda_{1}}\left(z\right)=0$. Since that 
\begin{align*}
w_{\lambda_{1}}>0\quad\text{in }\left(\Sigma_{\lambda_{1}}\backslash\Sigma''_{\lambda_{1}}\right)\cap B_{1}\left(2\lambda_{1}e_{1}\right),
\end{align*}
in a manner analogous to the calculations in case 1 of \textbf{Step 3} within the proof of Theorem \ref{T1-3}-(1), when $i=1$,
\begin{align*}
\mathcal{I}_{1}w_{\lambda_{1}}\left(z\right)
&=\text{P.V.}\int^{\lambda_{1}}_{-\infty}\left(\frac{1}{\left|t-z_{1}\right|^{1+2s}}-\frac{1}{\left|2\lambda_{1}-t-z_{1}\right|^{1+2s}}\right)w_{\lambda_{1}}\left(z_{t}\right)dt\\
&\geq \text{P.V.}\int^{-1}_{2\lambda_{1}-1}\left(\frac{1}{\left|t-z_{1}\right|^{1+2s}}-\frac{1}{\left|2\lambda_{1}-t-z_{1}\right|^{1+2s}}\right)w_{\lambda_{1}}\left(z_{t}\right)dt>0.
\end{align*}
When $i=2,...,N$, there is
\begin{align*}
\mathcal{I}_{i}w_{\lambda_{1}}\left(z\right)
=\text{P.V.}\int_{\R}\frac{w_{\lambda_{1}}\left(z+te_{i}\right)}{\left|t\right|^{1+2s}}dt>0.
\end{align*}
Therefore,
\begin{align*}
0=f\left(u_{\lambda}\right)\left(z\right)-f\left(u\right)\left(z\right)=-\mathcal{I}w_{\lambda_{1}}\left(z\right)<0,
\end{align*}
this cannot be true. 

For any given $\sigma>0$, in the bounded domain $\Sigma''_{\lambda_{1}-\sigma}$, we have
\begin{align*}
w_{\lambda_{1}}\geq C\left(\lambda_{1},\sigma\right)\quad\text{in }\overline{\Sigma''_{\lambda_{1}-\sigma}}.
\end{align*}
There exists $\epsilon_{\sigma}\in\left(0,-\lambda_{1}\right)$ such that for any $\lambda\in\left[\lambda_{1},\lambda_{1}+\epsilon_{\sigma}\right)$,             
\begin{align}\label{E5-2}
w_{\lambda}\geq0\quad\text{in }\Sigma''_{\lambda_{1}-\sigma}.
\end{align}
Let $D=\Sigma''_{\lambda}\backslash\Sigma''_{\lambda_{1}-\sigma}$. We find that $w_{\lambda}\in C^{2}\left(D\right)\cap\mathcal{L}_{s}$ satisfies
\begin{align*}
\left\lbrace 
\begin{aligned}
-\mathcal{I}w_{\lambda}-C\left(f\right)w_{\lambda}&\geq0&&\text{in }D,\\
w_{\lambda}&\geq0&&\text{in }\Sigma_{\lambda}\backslash D.
\end{aligned}
\right.
\end{align*}
By Lemma \ref{L2-7}, there exists $\sigma_{0}>0$ such that for any $\sigma\in\left(0,\sigma_{0}\right)$, we have $w_{\lambda}\geq0$ in $D$. We can choose $\sigma\in\left(0,\sigma_{0}\right)$ to get
\begin{align}\label{E5-3}
w_{\lambda}\geq0\quad\text{in }\Sigma''_{\lambda}\backslash\Sigma''_{\lambda_{1}-\sigma}.
\end{align}
From \eqref{E5-2} and \eqref{E5-3}, we have for any $\lambda\in\left[\lambda_{1},\lambda_{1}+\epsilon_{\sigma}\right)$,
\begin{align*}
w_{\lambda}\geq0\quad\text{in }\Sigma''_{\lambda},
\end{align*}
this contradicts the definition of $\lambda_{1}$. Thus we have $\lambda_{1}=0$. 

We can also move the plane $T_{\lambda}$ as $\lambda$ decreases from $1$ to $0$ and we define
\begin{align*}
\lambda_{2}:=\inf\left\{\lambda\geq0:w_{\eta}\leq0\text{ in }\Sigma''_{\eta},\eta\geq\lambda\right\}.
\end{align*}
Similarly, we get $\lambda_{2}=0$. We have $w_{\lambda_{1}}\geq0$ in $\Sigma''_{\lambda_{1}}$ and $w_{\lambda_{2}}\leq0$ in $\Sigma''_{\lambda_{2}}$, hence $w_{\lambda_{1}}=w_{\lambda_{2}}\equiv0$. Therefore $u$ is symmetric about $T_{0}$ and $u$ decreases in $e_{1}$ direction. We choose all the coordinate axis directions $e_{i}$ to obtain that $u$ is symmetric with respect to $P_{e_{i},0}$ and $u$ decreases in $e_{i}$ direction.

\vspace*{1em}
\noindent\textbf{Acknowledgements:} We would like to thank Professors Isabeau Birindelli and Giulio Galise at Sapienza Università di Roma for their kind guidance and suggestions. Lele Du thanks the financial support received from Sapienza Università di Roma through the projects BANDO AR n.9/2022 and BANDO AR n.1/2024. Minbo Yang is partially supported by National Natural Science Foundation of China (12471114) and Natural Science Foundation of Zhejiang Province (LZ22A010001).

\end{document}